\documentclass[12pt]{amsart}
\usepackage{amsfonts}
\usepackage{amssymb}
\usepackage{amsmath}
\usepackage{graphicx}
\usepackage{axodraw}  %Pour les arbres, les graphes, etc.
\usepackage{amsthm}
\usepackage{amscd}
\usepackage[all]{xy}               % pour les diagrammes
\usepackage{epsfig,exscale}  % pour les dessins
\usepackage{color}
\usepackage{txfonts} 
 \font \eightrm=cmr8
 
\newcommand{\nc}{\newcommand}
\newcommand{\butcher}{{\scriptstyle\circleright}}
\newcommand{\lbutcher}{{\raise 3.9pt\hbox{$\circ$}}\hskip -1.9pt{\scriptstyle \searrow}}

\newtheorem{thm}{Theorem}
\newtheorem{exam}{Example}

\newtheorem{cor}[thm]{Corollary}

\newtheorem{lem}[thm]{Lemma}

\newtheorem{defn}{Definition}

\nc{\ignore}[1]{{}}
\nc{\mrm}[1]{{\rm #1}}
\nc{\dirlim}{\displaystyle{\lim_{\longrightarrow}}\,}
\nc{\invlim}{\displaystyle{\lim_{\longleftarrow}}\,}
\nc{\vep}{\varepsilon} \nc{\ep}{\epsilon}
\nc{\sigmat}{\widetilde\sigma}
\nc{\ostar}{\overline{*}}

\nc{\mchar}{\mrm{Char}}
\nc{\Hom}{\mrm{Hom}}
\nc{\id}{\mrm{id}}

\nc{\remark}{\noindent{\bf{Remark:}}}
\nc{\remarks}{\noindent{\bf{Remarks:}}}

 \nc{\grad}[1]{^{({#1})}}
 \nc{\fil}[1]{_{#1}}

\nc{\BA}{{\Bbb A}} \nc{\CC}{{\Bbb C}} \nc{\DD}{{\Bbb D}}
\nc{\EE}{{\Bbb E}} \nc{\FF}{{\Bbb F}} \nc{\GG}{{\Bbb G}}
\nc{\HH}{{\Bbb H}} \nc{\LL}{{\Bbb L}} \nc{\NN}{{\Bbb N}}
\nc{\PP}{{\Bbb P}} \nc{\QQ}{{\Bbb Q}} \nc{\RR}{{\Bbb R}}
\nc{\TT}{{\Bbb T}} \nc{\VV}{{\Bbb V}} \nc{\ZZ}{{\Bbb Z}}
\nc{\Cal}[1]{{\mathcal {#1}}}
\nc{\mop}[1]{\mathop{\hbox {\rm #1} }}
\nc{\smop}[1]{\mathop{\hbox {\eightrm #1} }}
\nc{\mopl}[1]{\mathop{\hbox {\rm #1} }\limits}
\nc{\frakg}{{\frak g}}
\nc{\g}[1]{{\frak {#1}}}
\nc{\wt}{\widetilde}
\nc{\wh}{\widehat}
\nc{\un}{\hbox{\bf 1}}
\nc{\redtext}[1]{\textcolor{red}{#1}}
\nc{\bluetext}[1]{\textcolor{blue}{#1}}

\nc\fleche[1]{\mathop{\hbox to #1 mm{\rightarrowfill}}\limits}
\def\semi{\mathrel{\times}\kern -.85pt\joinrel\mathrel{\raise
    1.4pt\hbox{${\scriptscriptstyle |}$}}}

\def\fleche#1{\mathop{\hbox to #1 mm{\rightarrowfill}}\limits}
\def\gfleche#1{\mathop{\hbox to #1 mm{\leftarrowfill}}\limits}
\def\inj#1{\mathop{\hbox to #1 mm{$\lhook\joinrel$\rightarrowfill}}\limits}
\def\ginj#1{\mathop{\hbox to #1 mm{\leftarrowfill$\joinrel\rhook$}}\limits}
\def\surj#1{\mathop{\hbox to #1 mm{\rightarrowfill\hskip 2pt\llap{$\rightarrow$}}}\limits}
\def\gsurj#1{\mathop{\hbox to #1 mm{\rlap{$\leftarrow$}\hskip 2pt
      \leftarrowfill}}\limits}
\def \restr#1{\mathstrut_{\textstyle |}\raise-6pt\hbox{$\scriptstyle #1$}}
\def \srestr#1{\mathstrut_{\scriptstyle |}\hbox to
-1.5pt{}\raise-4pt\hbox{$\scriptscriptstyle #1$}}

\def\diagrama #1{\vskip 4mm \centerline {#1} \vskip 4mm}
\newcommand{\treel}{\hskip 0.5pc\scalebox{-0.3}{{\parbox{0.5pc}{
  \begin{picture}(60,45) (75,10)
    \SetWidth{1.5}
    \SetColor{Black}
    \Line(90,0)(90,45)
  \end{picture}}}}}

\newcommand{\treelsmall}{\hskip 0.5pc\scalebox{-0.3}{{\parbox{0.5pc}{
  \begin{picture}(60,45) (75,20)
    \SetWidth{1.5}
    \SetColor{Black}
    \Line(90,0)(90,45)
  \end{picture}}}}}

\newcommand{\treesmall}{\hskip 0.8pc\scalebox{-0.3}{{\parbox{0.5pc}{
  \begin{picture}(30,45) (75,-30)
    \SetWidth{1.5}
    \SetColor{Black}
   \Line(90,0)(70,-40)
    \Line(90,0)(110,-40)
    \Line(90,0)(90,15)
  \end{picture}}}}}

\newcommand{\treeA}{\hskip 1.5pc\scalebox{-0.3}{{\parbox{0.5pc}{
   \begin{picture}(60,75) (75,-30)
    \SetWidth{1.5}
    \SetColor{Black}
    \Line(90,0)(75,-30)
    \Line(90,0)(105,-30)
    \Line(105,30)(90,0)
    \Line(105,30)(135,-30)
    \Line(105,45)(105,30)
  \end{picture}}}}}

\newcommand{\treeB}{\hskip 1.5pc\scalebox{-0.3}{{\parbox{0.5pc}{
  \begin{picture}(60,75) (75,-30)
    \SetWidth{1.5}
    \SetColor{Black}
    \Line(90,0)(75,-30)
    \Line(105,30)(135,-30)
    \Line(105,45)(105,30)
    \Line(120,0)(105,-30)
    \Line(105,30)(90,0)
  \end{picture}}}}}

\newcommand{\treeC}{\hskip 1.5pc\scalebox{-0.2}{{\parbox{0.5pc}{
 \begin{picture}(90,105) (75,-30)
    \SetWidth{2.5}
    \SetColor{Black}
    \Line(90,0)(75,-30)
    \Line(120,60)(90,0)
    \Line(90,0)(105,-30)
    \Line(105,30)(135,-30)
    \Line(120,60)(165,-30)
    \Line(120,80)(120,60)
  \end{picture}
}}}}

\newcommand{\treeD}{\hskip 1.5pc\scalebox{-0.2}{{\parbox{0.5pc}{
 \begin{picture}(90,105) (75,-30)
    \SetWidth{2.5}
    \SetColor{Black}
    \Line(90,0)(75,-30)
    \Line(120,60)(90,0)
    \Line(90,0)(105,-30)
    \Line(120,60)(165,-30)
    \Line(120,80)(120,60)
    \Line(150,0)(135,-30)
  \end{picture}
}}}}

\newcommand{\treeE}{\hskip 1.5pc\scalebox{-0.2}{{\parbox{0.5pc}{
 \begin{picture}(90,105) (75,-30)
    \SetWidth{2.5}
    \SetColor{Black}
    \Line(90,0)(75,-30)
    \Line(120,60)(90,0)
    \Line(120,60)(165,-30)
    \Line(120,80)(120,60)
    \Line(150,0)(135,-30)
    \Line(135,30)(105,-30)
  \end{picture}
}}}}

\newcommand{\treeF}{\hskip 1.5pc\scalebox{-0.2}{{\parbox{0.5pc}{
 \begin{picture}(90,105) (75,-30)
    \SetWidth{2.5}
    \SetColor{Black}
    \Line(90,0)(75,-30)
    \Line(120,60)(90,0)
    \Line(120,60)(165,-30)
    \Line(120,80)(120,60)
    \Line(135,30)(105,-30)
    \Line(120,0)(135,-30)
  \end{picture}
}}}}

\newcommand{\treeG}{\hskip 1.5pc\scalebox{-0.2}{{\parbox{0.5pc}{
 \begin{picture}(90,105) (75,-30)
    \SetWidth{2.5}
    \SetColor{Black}
    \Line(90,0)(75,-30)
    \Line(120,60)(90,0)
    \Line(120,60)(165,-30)
    \Line(120,80)(120,60)
    \Line(105,30)(135,-30)
    \Line(120,0)(105,-30)
  \end{picture}
}}}}

\def\racine{{\scalebox{0.3}{
\begin{picture}(12,12)(38,-38)
\SetWidth{0.5} \SetColor{Black} \Vertex(45,-28){5.66}
\end{picture}}}}
  
 \def\arbrea{\,{\scalebox{0.15}{ 
  \begin{picture}(8,55) (370,-248)
    \SetWidth{2}
    \SetColor{Black}
    \Line(374,-244)(374,-200)
    \Vertex(374,-197){9}
    \Vertex(375,-245){12}
  \end{picture}
}}\,}

 \def\arbreba{\,{\scalebox{0.15}{ 
\begin{picture}(8,106) (370,-197)
    \SetWidth{2}
    \SetColor{Black}
    \Line(374,-193)(374,-149)
    \Vertex(374,-146){9}
    \Vertex(375,-194){12}
    \Line(374,-142)(374,-98)
    \Vertex(374,-95){9}
  \end{picture}
}}\,}

 \def\arbrebb{\,{\scalebox{0.15}{ 
  \begin{picture}(48,48) (349,-255)
    \SetWidth{2}
    \SetColor{Black}
    \Vertex(375,-252){12}
    \Line(376,-250)(395,-215)
    \Line(373,-251)(354,-214)
    \Vertex(353,-211){9}
    \Vertex(395,-213){9}
  \end{picture}
}}}

\def\arbreca{\,{\scalebox{0.15}{
\begin{picture}(8,156) (370,-147)
    \SetWidth{2}
    \SetColor{Black}
    \Line(374,-143)(374,-99)
    \Vertex(374,-96){9}
    \Vertex(375,-144){12}
    \Line(374,-92)(374,-48)
    \Vertex(374,-45){9}
    \Line(374,-42)(374,2)
    \Vertex(374,5){9}
  \end{picture}
}}\,}

\def\arbrecb{\,{\scalebox{0.15}{
\begin{picture}(48,94) (349,-255)
\SetWidth{2}
\SetColor{Black}
\Line(376,-204)(395,-169)
\Line(373,-205)(354,-168)
\Vertex(353,-165){9}
\Vertex(395,-167){9}
\Vertex(374,-205){9}
\Line(374,-246)(374,-209)
\Vertex(374,-252){12}
\end{picture}}}\,}

\def\arbrecc{\,{\scalebox{0.15}{
 \begin{picture}(48,98) (349,-205)
    \SetWidth{2}
    \SetColor{Black}
    \Vertex(375,-202){12}
    \Line(376,-200)(395,-165)
    \Line(373,-201)(354,-164)
    \Vertex(353,-161){9}
    \Vertex(395,-163){9}
    \Line(353,-160)(353,-113)
    \Vertex(353,-111){9}
  \end{picture}
}}\,}

\def\arbreccc{\,{\scalebox{0.15}{
 \begin{picture}(48,98) (349,-205)
    \SetWidth{2}
    \SetColor{Black}
    \Vertex(375,-202){12}
    \Line(376,-200)(395,-165)
    \Line(373,-201)(354,-164)
    \Vertex(353,-161){9}
    \Vertex(395,-163){9}
    \Line(395,-160)(395,-113)
    \Vertex(395,-111){9}
  \end{picture}
}}\,}

\def\arbrecd{\,{\scalebox{0.15}{
\begin{picture}(48,52) (349,-251)
    \SetWidth{2}
    \SetColor{Black}
    \Vertex(375,-248){12}
    \Line(376,-246)(395,-211)
    \Line(373,-247)(354,-210)
    \Vertex(353,-207){9}
    \Vertex(395,-209){9}
    \Line(375,-247)(375,-206)
    \Vertex(376,-203){9}
  \end{picture}
 }}\,}

\def\arbreda{\,{\scalebox{0.15}{
\begin{picture}(8,204) (370,-99)
    \SetWidth{2}
    \SetColor{Black}
    \Line(374,-95)(374,-51)
    \Vertex(374,-48){9}
    \Vertex(375,-96){12}
    \Line(374,-44)(374,0)
    \Vertex(374,3){9}
    \Line(374,6)(374,50)
    \Vertex(374,53){9}
    \Line(374,53)(374,98)
    \Vertex(374,101){9}
  \end{picture}
}}\,}

\def\arbredb{\,{\scalebox{0.15}{
\begin{picture}(48,135) (349,-255)
    \SetWidth{2}
    \SetColor{Black}
    \Line(376,-163)(395,-128)
    \Line(373,-164)(354,-127)
    \Vertex(353,-124){9}
    \Vertex(395,-126){9}
    \Vertex(374,-164){9}
    \Line(374,-205)(374,-168)
    \Vertex(374,-207){9}
    \Line(374,-248)(374,-211)
    \Vertex(374,-252){12}
  \end{picture}
}}\,}

\def\arbredc{\,{\scalebox{0.15}{
 \begin{picture}(48,150) (349,-205)
    \SetWidth{2}
    \SetColor{Black}
    \Line(376,-148)(395,-113)
    \Line(373,-149)(354,-112)
    \Vertex(353,-109){9}
    \Vertex(395,-111){9}
    \Line(353,-108)(353,-61)
    \Vertex(353,-59){9}
    \Line(374,-200)(374,-153)
    \Vertex(374,-149){9}
    \Vertex(374,-202){12}
  \end{picture}
}}\,}

\def\arbredd{\,{\scalebox{0.15}{
 \begin{picture}(48,99) (349,-251)
    \SetWidth{2}
    \SetColor{Black}
    \Line(376,-199)(395,-164)
    \Line(373,-200)(354,-163)
    \Vertex(353,-160){9}
    \Vertex(395,-162){9}
    \Vertex(376,-156){9}
    \Vertex(376,-248){12}
    \Line(375,-245)(375,-204)
    \Line(375,-200)(375,-159)
    \Vertex(375,-201){9}
  \end{picture}
}}\,}

\def\arbrede{\,{\scalebox{0.15}{
 \begin{picture}(48,153) (349,-150)
    \SetWidth{2}
    \SetColor{Black}
    \Vertex(375,-147){12}
    \Line(376,-145)(395,-110)
    \Line(373,-146)(354,-109)
    \Vertex(353,-106){9}
    \Vertex(395,-108){9}
    \Line(353,-105)(353,-58)
    \Vertex(353,-56){9}
    \Line(353,-52)(353,-5)
    \Vertex(353,-1){9}
  \end{picture}
}}\,}

\def\arbredf{\,{\scalebox{0.15}{
\begin{picture}(48,98) (349,-205)
    \SetWidth{2}
    \SetColor{Black}
    \Vertex(375,-202){12}
    \Line(376,-200)(395,-165)
    \Line(373,-201)(354,-164)
    \Vertex(353,-161){9}
    \Vertex(395,-163){9}
    \Line(353,-160)(353,-113)
    \Vertex(353,-111){9}
    \Line(395,-159)(395,-112)
    \Vertex(395,-111){9}
  \end{picture}
}}\,}

\def\arbredz{\,{\scalebox{0.15}{
  \begin{picture}(68,88) (329,-215)
    \SetWidth{2}
    \SetColor{Black}
    \Vertex(375,-212){12}
    \Line(376,-210)(395,-175)
    \Line(373,-211)(354,-174)
    \Vertex(353,-171){9}
    \Vertex(395,-173){9}
    \Line(351,-168)(332,-131)
    \Line(355,-168)(374,-133)
    \Vertex(333,-131){9}
    \Vertex(374,-131){9}
  \end{picture}
}}\,}

\def\arbredg{\,{\scalebox{0.15}{
\begin{picture}(48,98) (349,-205)
    \SetWidth{2}
    \SetColor{Black}
    \Vertex(375,-202){12}
    \Line(376,-200)(395,-165)
    \Line(373,-201)(354,-164)
    \Vertex(353,-161){9}
    \Vertex(395,-163){9}
    \Line(375,-201)(375,-160)
    \Vertex(376,-157){9}
    \Vertex(376,-111){9}
    \Line(375,-155)(375,-114)
  \end{picture}
}}\,}

\def\arbredh{\,{\scalebox{0.15}{
 \begin{picture}(90,46) (330,-257)
    \SetWidth{2}
    \SetColor{Black}
    \Vertex(375,-254){12}
    \Line(376,-252)(395,-217)
    \Vertex(395,-215){9}
    \Line(374,-254)(335,-226)
    \Vertex(334,-224){9}
    \Line(375,-252)(356,-215)
    \Vertex(355,-215){9}
    \Line(374,-255)(417,-227)
    \Vertex(418,-225){9}
  \end{picture}
}}\,}

\def\arbrebbLab{\,{\scalebox{0.22}{
  \begin{picture}(48,48) (349,-255)
    \SetWidth{2}
    \SetColor{Black}
    \Vertex(375,-252){12}
    \Line(376,-250)(395,-215)
    \Line(373,-251)(354,-214)
    \Vertex(353,-211){9}
    \Vertex(395,-211){9}
    \Text(370,-285)[lb]{\Huge{\Black{$u$}}}
    \Text(320,-215)[lb]{\Huge{\Black{$w$}}}
    \Text(415,-215)[lb]{\Huge{\Black{$v$}}}
  \end{picture}
}}}

\def\arbrebbLabel{\,{\scalebox{0.25}{
  \begin{picture}(48,48) (349,-255)
    \SetWidth{2}
    \SetColor{Black}
    \Vertex(375,-252){12}
    \Line(376,-250)(395,-215)
    \Line(373,-251)(354,-214)
    \Vertex(353,-211){9}
    \Vertex(395,-211){9}
    \Text(370,-290)[lb]{\Huge{\Black{$v_1$}}}
    \Text(320,-215)[lb]{\Huge{\Black{$v_2$}}}
    \Text(408,-215)[lb]{\Huge{\Black{$v_3$}}}
  \end{picture}
}}}

\def\arbrebbLabe{\,{\scalebox{0.25}{
  \begin{picture}(48,48) (349,-255)
    \SetWidth{2}
    \SetColor{Black}
    \Vertex(375,-252){12}
    \Line(376,-250)(395,-215)
    \Line(373,-251)(354,-214)
    \Vertex(353,-211){9}
    \Vertex(395,-211){9}
    \Text(370,-290)[lb]{\Huge{\Black{$w_1$}}}
    \Text(315,-215)[lb]{\Huge{\Black{$w_3$}}}
    \Text(408,-215)[lb]{\Huge{\Black{$w_2$}}}
  \end{picture}
}}}

\def\arbrecccLab{\,{\scalebox{0.25}{
 \begin{picture}(48,98) (349,-205)
    \SetWidth{2}
    \SetColor{Black}
    \Vertex(375,-202){12}
    \Line(376,-200)(395,-165)
    \Line(373,-201)(354,-164)
    \Vertex(353,-161){9}
    \Vertex(395,-163){9}
    \Line(395,-160)(395,-113)
    \Vertex(395,-111){9}
    \Text(370,-235)[lb]{\Huge{\Black{$v_1$}}}
    \Text(320,-170)[lb]{\Huge{\Black{$v_3$}}}
    \Text(410,-170)[lb]{\Huge{\Black{$v_2$}}}
    \Text(410,-120)[lb]{\Huge{\Black{$v_4$}}}
  \end{picture}
}}\,}

\def\arbrecdLab{\,{\scalebox{0.25}{
\begin{picture}(48,52) (349,-251)
    \SetWidth{2}
    \SetColor{Black}
    \Vertex(375,-248){12}
    \Line(376,-246)(395,-211)
    \Line(373,-247)(354,-210)
    \Vertex(353,-207){9}
    \Vertex(395,-209){9}
    \Line(375,-247)(375,-206)
    \Vertex(376,-203){9}
    \Text(370,-285)[lb]{\Huge{\Black{$w_1$}}}
    \Text(315,-210)[lb]{\Huge{\Black{$w_4$}}}
    \Text(365,-190)[lb]{\Huge{\Black{$w_3$}}}
    \Text(410,-210)[lb]{\Huge{\Black{$w_2$}}}
  \end{picture}
 }}\,}

\def\arbreebz{\,{\scalebox{0.25}{
\begin{picture}(90,120) (350,-275)
    \SetWidth{2}
    \SetColor{Black}
    \Vertex(400,-290){12}
    \Line(470,-130)(470,-210)
    \Vertex(470,-130){9}
    \Line(398,-299)(330,-210)
    \Vertex(330,-210){9}
    \Line(320,-205)(270,-130)
    \Vertex(270,-130){9}
    \Line(400,-299)(470,-210)
    \Vertex(470,-210){9}
    \Line(335,-210)(370,-130)
    \Vertex(370,-130){9}
    \Line(475,-210)(520,-130)
    \Vertex(520,-130){9}
    \Line(465,-210)(420,-130)
    \Vertex(420,-130){9}
    \Line(520,-130)(520,-40)
    \Vertex(520,-40){9}
    \Line(270,-130)(270,-40)
    \Vertex(270,-40){9}
    \Text(397,-325)[lb]{\Huge{\Black{$1$}}}
    \Text(465,-245)[lb]{\Huge{\Black{$2$}}}
    \Text(325,-245)[lb]{\Huge{\Black{$7$}}}
    \Text(250,-120)[lb]{\Huge{\Black{$9$}}}
    \Text(465,-115)[lb]{\Huge{\Black{$5$}}}
    \Text(365,-115)[lb]{\Huge{\Black{$8$}}}
    \Text(530,-120)[lb]{\Huge{\Black{$3$}}}
    \Text(415,-115)[lb]{\Huge{\Black{$6$}}}
    \Text(515,-25)[lb]{\Huge{\Black{$4$}}}
    \Text(260,-25)[lb]{\Huge{\Black{$10$}}}
  \end{picture}}}\,}
\def\table{{\scalebox{1.00}{
\begin{picture}(275,275) (290,170)
\SetWidth{0.8}
\SetColor{Black}
\Line(260,422)(610,422)
\Line(290,440)(290,170)
\Line(340,440)(340,170)
\Line(485,440)(485,170)
\Text(360,429)[lb]{\normalsize{\Black{$\Psi(\sigma)$}}}
\Text(500,429)[lb]{\normalsize{\Black{$\Psi^{-1}(\sigma)$}}}
\Text(310,429)[lb]{\normalsize{\Black{$d(\sigma)$}}}
\Text(270,432)[lb]{\normalsize{\Black{$\sigma$}}}
\Text(270,400)[lb]{\normalsize{\Black{$\racine$}}}
\Text(315,400)[lb]{\normalsize{\Black{$0$}}}
\Text(370,400)[lb]{\normalsize{\Black{$\racine$}}}
\Text(515,400)[lb]{\normalsize{\Black{$\racine$}}}
\Text(270,375)[lb]{\normalsize{\Black{$\arbrea$}}}
\Text(315,375)[lb]{\normalsize{\Black{$1$}}}
\Text(370,375)[lb]{\normalsize{\Black{$\arbrea$}}}
\Text(515,375)[lb]{\normalsize{\Black{$\arbrea$}}}
\Text(270,345)[lb]{\normalsize{\Black{$\arbreba$}}}
\Text(315,345)[lb]{\normalsize{\Black{$3$}}}
\Text(370,345)[lb]{\normalsize{\Black{$\arbreba$}}}
\Text(515,345)[lb]{\normalsize{\Black{$\arbreba$}}}
\Text(267,325)[lb]{\normalsize{\Black{$\arbrebb$}}}
\Text(315,325)[lb]{\normalsize{\Black{$2$}}}
\Text(360,325)[lb]{\normalsize{\Black{$\arbrebb+\arbreba$}}}
\Text(500,325)[lb]{\normalsize{\Black{$\arbrebb-\arbreba$}}}
\Text(270,290)[lb]{\normalsize{\Black{$\arbreca$}}}
\Text(315,290)[lb]{\normalsize{\Black{$6$}}}
\Text(370,290)[lb]{\normalsize{\Black{$\arbreca$}}}
\Text(515,290)[lb]{\normalsize{\Black{$\arbreca$}}}
\Text(267,265)[lb]{\normalsize{\Black{$\arbrecb$}}}
\Text(315,265)[lb]{\normalsize{\Black{$5$}}}
\Text(360,265)[lb]{\normalsize{\Black{$\arbrecb+\arbreca$}}}
\Text(500,265)[lb]{\normalsize{\Black{$\arbrecb-\arbreca$}}}
\Text(267,235)[lb]{\normalsize{\Black{$\arbrecc$}}}
\Text(315,235)[lb]{\normalsize{\Black{$4$}}}
\Text(360,235)[lb]{\normalsize{\Black{$\arbrecc+\arbreca$}}}
\Text(500,235)[lb]{\normalsize{\Black{$\arbrecc-\arbreca$}}}
\Text(267,205)[lb]{\normalsize{\Black{$\arbreccc$}}}
\Text(315,205)[lb]{\normalsize{\Black{$4$}}}
\Text(350,205)[lb]{\normalsize{\Black{$\arbreccc+\arbrecb+\arbreca$}}}
\Text(500,205)[lb]{\normalsize{\Black{$\arbreccc-\arbrecb$}}}
\Text(267,175)[lb]{\normalsize{\Black{$\arbrecd$}}}
\Text(315,175)[lb]{\normalsize{\Black{$3$}}}
\Text(350,175)[lb]{\normalsize{\Black{$\arbrecd+\arbrecc+2\arbreccc+\arbrecb+\arbreca$}}}
\Text(500,175)[lb]{\normalsize{\Black{$\arbrecd-\arbrecc-2\arbreccc+\arbrecb+\arbreca$}}}
\end{picture}}}}
\begin{document}
%%%%%%%%%%%%%%%%%%%%%%%%%%%%%%%%%%%%%%%%%%%%%
%%%%%%%%%%%%%%%%%%%%%%%%%%%%%%%%%%%%%%%%%%%%%
%%%%%%%%%%%%%%%%%%%%%%%%%%%%%%%%%%%%%%%%%%%%%
\title[Monomial bases]
      {Monomial bases for free pre-Lie algebras}

\author{Mahdi J. Hasan Al-Kaabi}

\address {Mathematics Department, College of Science, Al-Mustansiriya University, Palestine Street, P.O.Box 14022, Baghdad, IRAQ.}  
\address{Laboratoire de Math\'ematiques, CNRS-UMR6620, Universit\'e Blaise Pascal, 24 Avenue des Landais, Les C\'ezeaux, BP 80026, F63171 Aubi\`ere, CEDEX, France.}
\address{E-mail address:Mahdi.Alkaabi@math.univ-bpclermont.fr}

%%%%%%%%%%%%%%%%%%%%%%%%%%%%%%%%%%%%%%%%%%%%%%%%%%%%%%%%%%%%%%%%%%%
\date{}
%%%%%%%%%%%%%%%%%%%%%%%%%%%%%%%%%%%%%%%%%%%%%%%%%%%%%%%%%%%%%%%%%%%
\begin{abstract}
In this paper, we study the concept of free pre-Lie algebra generated by a (non-empty) set. We review the construction by A. Agrachev and R. Gamkrelidze \cite{AR81} of monomial bases in free pre-Lie algebras. We describe the matrix of the monomial basis vectors in terms of the rooted trees basis exhibited by F. Chapoton and M. Livernet in \cite{CL01}. Also, we show that this matrix is unipotent and we find an explicit expression for its coefficients, which uses a similar procedure for the free magmatic algebra at the level of planar rooted trees which has been suggested by K. Ebrahimi-Fard and D. Manchon.       
\end{abstract}

\maketitle
\tableofcontents

%%%%%%%%%%%%%%%%%%%%%%%%%%%%%%%%%%%%%%%%%%%%%%%%%%%%%%%%%%%%%%%%%%%
%%%%%%%%%%%%%%%%%%%%%%%%%%%%%%%%%%%%%%%%%%%%%%%%%%%%%%%%%%%%%%%%%%%
\section{Introduction}

Pre-Lie algebra structures appear in various domains of mathematics: differential geometry, quantum field theory, differential equations. They have been studied intensively recently; we refer e.g. to the survey papers: \cite{D.B06, P.C11, D.M}. Free pre-Lie algebras had already been studied as early as 1981 by A. Agrachev and R. V. Gamkrelidze in their joint work "\textsl{Chronological algebras and nonstationary vector fields}" \cite{AR81}, and also by D. Segal in \cite{DS94}. In particular, both papers give a construction of monomial bases, with different approaches. Besides, rooted trees are a classical topic, closely connected to pre-Lie algebras.  They appeared for example in the study of \textsl{vector fields} \cite{A.C57}, \textsl{numerical analysis} \cite{CB00}, and more recently in \textsl{quantum field theory} \cite{CK98}. Bases for free pre-Lie algebras in terms of rooted trees were introduced by F. Chapoton and M. Livernet in \cite{CL01}, using the point of view of operads. A. Dzhumadil'Daev and C. L\"ofwall described independently two bases for free pre-Lie algebras, one using the concept of rooted trees, and the other obtained by considering a basis for the free (non-associative) algebra modulo the pre-Lie relation \cite{AC02}. \\
 
In our paper, we study free pre-Lie algebras. We describe an explicit method for finding suitable monomial bases for them: recall that the space $\Cal{T}\!$ spanned by (non-planar) rooted trees forms with the grafting operation $"\to"$ the free pre-Lie algebra with one generator \cite{CL01, AC02}. A monomial in the free pre-Lie algebra with one generator is a parenthesized word built up from the generator $"\racine\,"$ and the pre-Lie grafting operation $"\to"$, for example:
$$(\racine\to\racine)\to\big(\racine\to(\racine\to\racine)\big).$$

We are interested in particular monomial bases which will be called "tree-grounded". To each monomial we can associate a "lower-energy term" by replacing the grafting operation "$\to$" by the Butcher product "$\butcher$". A monomial basis of $\Cal T$ will be called "tree-grounded" if the lower-energy terms of each monomial give back the Chapoton-Livernet tree basis of $\Cal T$. We show that tree-grounded monomial bases are in one-to-one correspondence with choices $t\mapsto S(t)$ of a planar representative for each tree $t$. We give an explicit expression for the coefficients of these monomials in the basis of rooted trees, thus exhibiting a square matrix $[\![\beta_S(s,t)]\!]_{s,\,t\in\Cal{T}_{\!\!n}}^{}$ for each degree $n>0$.\\
 
This paper consists in two main sections: Section \ref{section 1} contains some preliminaries on planar and non-planar rooted trees, Butcher products and grafting products. In this section, we also review the joint work of K. Ebrahimi-Fard and D. Manchon (unpublished) who described an explicit algebra isomorphism $\Psi$ between two structures of free magmatic algebras defined on the space $\Cal{T}^{pl}$ of all planar rooted trees, by the left Butcher product $"\,\lbutcher"$ and the left grafting product $"\searrow"$ respectively. We give the explicit expression of the coefficients $c(\sigma, \tau)$ of this isomorphism in the planar rooted tree basis. Using their work, and by defining a bijective linear map $\widetilde{\Psi}_S$ which depends on the choice of a section $S$ of the "forget planarity" projection $\pi$ alluded to above, we find a formula for the coefficients $\beta_S(s, t)$ of $\widetilde{\Psi}_S$ in the (non-planar) rooted tree basis. This can be visualized on the following diagram:
\diagrama{
\xymatrix{
\tau=m(\racine,\lbutcher) \in \Cal {T}_{\!\!n}^{pl} \ar[r]^{\Psi}\ar@<12mm>@{->>}[d]^{\pi}
&\Cal {T}_{\!\!n}^{pl} \ni m(\racine, \searrow)\ar@<-9mm>@{->>}[d]^{\pi}\\
t=m(\racine, \butcher) \in \Cal{T}_{\!\!n}\ar@{.>}[r]_{{\widetilde\Psi}_S}\ar@<-10mm>@{.>}[u]^S & \Cal{T}_{\!\!n}\ni m(\racine, \to)}}
for any homogeneous components $\Cal{T}^{pl}_{\!\!n}$ and $\Cal{T}_{\!\!n}$.\\

In Section \ref{free pre-lie algebra}, we recall some basic topics on free pre-Lie algebras. We describe the construction of a monomial basis for each homogeneous subspace $\Cal{A}_n$ in free pre-Lie algebras $\Cal{A}_E$ generated by a (non-empty) set $E$, using a type of algebra isomorphism obtained by A. Agrachev and R. V. Gamkrelidze \cite{AR81}. Finally, the constructions in Sections \ref{section 1} and \ref{free pre-lie algebra} can be related as follows: we show that a tree-grounded monomial basis in a free pre-Lie algebra defines a section $S$ of the projection $\pi:\Cal{T}^{pl} \longrightarrow\hskip -5mm\longrightarrow \Cal{T}$ and, conversely, that any section of $\pi$ defines a tree-grounded monomial basis.
     
\section{Planar and non-planar rooted trees}\label{section 1}

In graph theory, a tree is a undirected connected graph together with vertices connected with each other, without cycles, by simple paths called edges \cite{EM}. A rooted tree is defined as a tree with one designated vertex called the root. The other remaining vertices are partitioned into $k\geq 0$ disjoint subsets such that each of them in turn represents a rooted tree, and a subtree of the whole tree. This can be taken as a recursive definition for rooted trees, widely used in computer algorithms \cite{D.K68}. Rooted trees stand among the most important structures appearing in many branches of pure and applied mathematics. \\

In general, a tree structure can be described as a "branching" relationship between vertices, much like that found in the trees of nature. Many types of trees defined by all sorts of constraints on properties of vertices appear to be of interest in combinatorics and in related areas such as formal logic and computer science.

\begin{defn}
A planar binary tree is a finite oriented tree embedded in the plane, such that each internal vertex has exactly two incoming edges and one outgoing edge. One of the internal vertices, called the root, is a distinguished vertex with two incoming edges and one edge, like a tail at the bottom, not ending at a vertex.
\end{defn} 

The incoming edges in this type of trees are internal (connecting two internal vertices), or external (with one free end). The external incoming edges are called the leaves. We give here some examples of planar binary trees:
$$\treel\,\,\,\,\treesmall\,\,\,\,\,\,\treeA\,\,\,\,\treeB\,\,\,\,\treeC\,\,\,\,\treeE\,\,\,\,\treeD\,\,\,\,\treeF\,\,\,\,\treeG\,\,\,\,\ldots,$$
where the single edge "$\treelsmall$" is the unique planar binary tree without internal vertices. The degree of any planar binary tree is the number of its leaves. Denote by $T^{bin}_{pl}$ (respectively $\Cal{T}^{bin}_{\!pl}$) the set (respectively the linear span) of planar binary trees.\\

Define the grafting operation "$\vee$" on the space $\Cal{T}^{bin}_{\!pl}$ to be the operation that maps any planar binary trees $t_1, t_2$ into a new planar binary tree "$t_1 \vee t_2$", which takes the $Y$-shaped tree $\treesmall$ replacing the left (respectively the right) branch by $t_1$ (respectively $t_2$), see the following examples:
$$ \treel \vee \treel = \treesmall\,,\,\,\,\treel \vee \treesmall = \treeA\,,\,\,\,\treesmall \vee \treel = \treeB\,,\,\,\,\treesmall \vee \treesmall = \treeD\,,\,\,\,\treel \vee \treeA = \treeC.$$

Let $D$ be any (non-empty) set, the free magma $M_{\!D}^{}$ generated by $D$ can be described to be the set of planar binary trees with leaves decorated by the elements of $D$, together with the "$\vee$" product described above \cite{D.K68, PT09}. Moreover, the linear span $\Cal{T}^{bin}_{\!pl}$, generated by the trees of the magma $M_{\!D}^{}$ defined above, equipped with the grafting "$\vee$" is a description of the free magmatic algebra.   
\begin{defn}
For any positive integer $n$, a rooted tree of degree $n$, or simply $n$-rooted tree, is a finite oriented tree together with $n$ vertices. One of them, called the root, is a distinguished vertex without any outgoing edge. Any vertex can have arbitrarily many incoming edges, and any vertex distinct from the root has exactly one outgoing edge. Vertices with no incoming edges are called leaves.
\end{defn}

A rooted tree is said to be planar, if it is endowed with an embedding in the plane. Otherwise, its called a (non-planar) rooted tree. Here are the planar rooted trees up to four vertices:
$$\racine\,\,\,\,\arbrea\,\,\,\,\arbreba\,\,\,\,\arbrebb\,\,\,\,\arbreca\,\,\,\,\arbrecb\,\,\,\,\arbrecc\,\,\,\,\arbreccc\,\,\,\,\arbrecd\,\,\,\,\cdots $$

Denote by $T^{pl}$ (respectively $T$) the set of all planar (respectively non-planar) rooted trees, and $\Cal{T}^{pl}$ (respectively $\Cal{T}$) the linear space spanned by the elements of $T^{pl}$ (respectively $T$). Every rooted tree $\sigma$ can be written as:
\begin{equation}\label{B+}
\sigma = B_+(\sigma_1\,\cdots\,\sigma_k),
\end{equation}
where $B_+$ is the operation which grafts a monomial $\sigma_1\, \cdots\,\sigma_k$ of rooted trees on a common root, which gives a new rooted tree by connecting the root of each $\sigma_{i}$, by an edge, to the new root. The planar rooted tree $\sigma$ in formula \eqref{B+} depends on the order of the branch planar trees $\sigma_j$, whereas this order is not important for the corresponding (non-planar) tree.

\subsection{(Left) Butcher product and left grafting}
 The (left) Butcher product $"\,\lbutcher"$ of any planar rooted trees $\sigma$ and $\tau$ is:
\begin{equation}\label{butcher product}
 \sigma \lbutcher \tau := B_{+}(\sigma \tau_1 \cdots \tau_{k}),
\end{equation}
where $\tau_1, \ldots, \tau_{k}\!\in\!T^{pl}\!,$ such that $\tau = B_{+}(\tau_1 \cdots \tau_{k})$. It maps the pair of trees $(\sigma, \tau)$ into a new planar rooted tree induced by grafting the root of $\sigma$, on the left via a new edge, on the root of $\tau$.\\

 The usual product $"\butcher"$ in the non-planar case, given by the same formula \eqref{butcher product}, is known as the Butcher product. It is non-associative permutative (NAP), i.e. it satisfies the following identity:
$$ s \butcher (s' \butcher t) = s' \butcher ( s \butcher t),$$ 
for any (non-planar) trees $s,s',t$. Indeed, for $t = B_{+}(t_1 \cdots t_k)$, where
$t_1,..., t_k$ in $T$, we have:
\begin{align*}
s \butcher (s' \butcher t) &= s \butcher (B_{+}(s' t_1 \cdots t_k))\\
&= B_{+}(s s't_1 \cdots t_k)\\
&= B_{+}(s'\!s\,t_1 \cdots t_k)\\
&= s' \butcher (B_{+}(s\,t_1 \cdots t_k))\\ 
&= s' \butcher (s \butcher t).
\end{align*}

D. E. Knuth in his work \cite{D.K68} described a relation between the planar binary trees and the planar rooted trees. He introduced a bijection $\Phi: T^{bin}_{pl} \longrightarrow T^{pl}$ called the rotation correspondence\footnote{For more details about the rotation correspondence see \cite[Paragraph 2.3.2]{D.K68}, \cite{JM04} and \cite[Paragraph 1.5.3]{PT09}.}, recursively defined by:
\begin{equation}\label{Phi}
\Phi(\treel\!)= \racine\,,\,\,and\,\,\Phi(t_1 \vee t_2) = \Phi(t_1)\,\lbutcher\,\Phi(t_2), \forall t_1, t_2 \in T^{bin}_{pl}.
\end{equation}
Let us compute a few terms:
$$\Phi(\treesmall\,) = \Phi(\treel)\,\lbutcher\,\Phi(\treel) = \arbrea\,,\,\,\,\Phi(\treeB) = \Phi(\treesmall\,)\,\lbutcher\,\Phi(\treel) = \arbreba\,,\,\,\,\Phi(\treeA) = \arbrebb, $$
$$\Phi(\treeE) = \arbreca\,,\,\,\,\Phi(\treeF) = \arbrecb\,,\,\,\,\Phi(\treeD) = \arbrecc\,,\,\,\,\Phi(\treeG) = \arbreccc\,,\,\,\,\Phi(\treeC) = \arbrecd.$$

The bijection given in \eqref{Phi} realizes the free magma $M_{\!D}^{}$ as the set of planar rooted trees with $D$-decorated vertices, endowed with the left Butcher product. Also, the linear span $\Cal{T}^{pl},$ generated by the planar trees of the magma $M_{\!D}^{}$, forms with the product $"\,\lbutcher"$ another description of the free magmatic algebra.
 
\begin{defn}\label{left grafting}
The left grafting $"\searrow"$ is a bilinear operation defined on the vector space $\Cal{T}^{pl}$, such that for any planar rooted trees $\sigma$ and $\tau$:
\begin{equation}\label{searrow}
\sigma \searrow \tau = \sum_{v\,vertex\,of\,\tau}{\sigma \searrow_{v}\tau},
\end{equation}
where "$\sigma \searrow_{v} \tau$" is the tree obtained by grafting the tree $\sigma$, on the left, on the vertex $v$ of the tree $\tau$, such that $\sigma$ becomes the leftmost branch, starting from $v$, of this new tree.
\end{defn}

\begin{exam}
$$\arbrea \searrow \arbreba = \arbreda+\arbredc+\arbredf.$$
\end{exam}
This type of grafting again provides the space $\Cal{T}^{pl}$ with a structure of free magmatic algebra: K. Ebrahimi-Fard and D. Manchon showed that the two structures defined on $\Cal{T}^{pl}$, one by the product $"\,\lbutcher"$ and the other by $"\searrow"$, are linearly isomorphic, as follows: define the potential energy $d(\sigma)$ of a planar rooted tree $\sigma$ to be the sum of the heights of its vertices. Introduce the decreasing filtration $\Cal{T}^{pl} = \Cal{T}^{(0)}_{pl} \supset \Cal{T}^{(1)}_{pl} \supset \Cal{T}^{(2)}_{pl} \supset ...$, where $\Cal{T}^{(k)}_{pl}$ is the vector space spanned by planar rooted trees $\sigma$ with $d(\sigma) \geq k$.  
\begin{thm}\label{D.K}
There is a unique linear isomorphism $\Psi$ from $\Cal{T}^{pl}$ onto $\Cal{T}^{pl}$, defined as:
\begin{equation}
\Psi(\racine\,) = \racine\,,\,and\,\,\Psi(\sigma_1 \lbutcher\,\sigma_2) = \Psi(\sigma_1)\searrow\Psi(\sigma_2),\, \forall \sigma_1, \sigma_2 \in T^{pl}.
\end{equation} 
It respects the graduation (given by the number of vertices), and the associated graded map $Gr\,\Psi$ (with respect to the potential energy filtration above) reduces to the identity.
\end{thm}
\begin{proof}
The linear map $\Psi$ is uniquely determined by virtue of the universal property of the free magmatic algebra $(\Cal{T}^{pl},\lbutcher)$, and obviously respects the number of vertices. For any planar rooted trees $\sigma_1, \sigma_2$, the equality $\sigma_1 \searrow \sigma_2 = \sigma_1 \lbutcher\,\sigma_2 + \sigma'$ holds, with $\sigma' \in \Cal{T}^{(d(\sigma_1 \!\!\lbutcher\,\sigma_2)+1)}_{pl}$. Then, for $\sigma = \sigma_1 \lbutcher\,\sigma_2$, we have:
\begin{equation}
\Psi(\sigma) = \sigma + \sigma''\!,
\end{equation} 
with $\sigma'' \in \Cal{T}^{(d(\sigma)+1)}_{pl}$, which proves the Theorem.
\end{proof}
From Theorem \ref{D.K}, one can note that the matrix of $\,\Psi$ restricted to any homogeneous component $\Cal{T}^{pl}_{n}$ is upper triangular unipotent. More precisely, $c(\sigma,\tau)=0$ if the potential energy $d(\sigma)$ of $\sigma$ is strictly smaller than the potential energy of $\tau$, and $c(\sigma,\tau)=\delta^{\tau}_{\sigma}$ if $ d(\sigma)=d(\tau)$. We can calculate the sum of the entries of this matrix as follows: for any planar rooted tree $\sigma \in T^{pl}_n$, let $N(\sigma)$ be the number of trees (with the multiplicities) in $\Psi(\sigma)$. Let $\sigma=\sigma_1 \lbutcher \sigma_2$, where $\sigma_1 \in T^{pl}_p,\ \sigma_2 \in T^{pl}_q,\hbox{ such that } p+q=n, \hbox{ for } p,\ q \geq 1$. Since $\sigma_2$ has $q$  vertices, and from the definition of the left grafting product $"\searrow"$, we get that:
\begin{equation}\label{eq.4}
N(\sigma)=N(\sigma_1)\ N(\sigma_2)\ q.
\end{equation}

Now, define: 
\begin{equation}\label{eq.5}
N(T^{pl}_n)=\sum_{\sigma \in T^{pl}_n} {N(\sigma)},
\end{equation}
then using \eqref{eq.4}, we obtain that:
\begin{align*}
N(T^{pl}_n) &= \sum_ {{\scriptstyle p+q=n} \atop {\scriptstyle p,\ q \geq 1}}{\sum_ {{\scriptstyle \sigma_1 \in T^{pl}_p} \atop {\scriptstyle \sigma_2 \in T^{pl}_q}} {N(\sigma_1)\ N(\sigma_2)\ q}}\\
&= \sum_ {{\scriptstyle p+q=n} \atop {\scriptstyle p,\ q \geq 1}}\ q \left(\sum_ {\sigma_1 \in T^{pl}_p}{N(\sigma_1)}\right)\ \left(\sum_ {\sigma_2 \in T^{pl}_q}{N(\sigma_2)}\right) \\
&= \sum_ {{\scriptstyle p+q=n} \atop {\scriptstyle p,\ q \geq 1}} {N(T^{pl}_p)\ N(T^{pl}_q)\ q}.
\end{align*}
Here, we find some terms of the formula \eqref{eq.5}:
\begin{align*}
N(T^{pl}_1) &= N(T^{pl}_2) = 1\\
N(T^{pl}_3) &= N(T^{pl}_2)\ N(T^{pl}_1)\ 1 + N(T^{pl}_1)\ N(T^{pl}_2)\ 2 = 3\\ 
N(T^{pl}_4) &= N(T^{pl}_3)\ N(T^{pl}_1)\ 1 + N(T^{pl}_2)\ N(T^{pl}_2)\ 2 + N(T^{pl}_1)\ N(T^{pl}_3)\ 3 = 14\\
N(T^{pl}_5) &= N(T^{pl}_4)\ N(T^{pl}_1)\ 1 + N(T^{pl}_3)\ N(T^{pl}_2)\ 2 + N(T^{pl}_2)\ N(T^{pl}_3)\ 3 + N(T^{pl}_1)\ N(T^{pl}_4)\ 4 = 85.
\end{align*}

This is sequence $A088716$ in \cite{S}. The generating series $A(x):=\sum\limits_{n \geq 1}{a_nx^n}$, modulo the shift $a_n:=N(T^{pl}_{n+1})$, verifies the differential equation:
$$A(x)=1+xA(x)^2+x^2A(x)A'(x).$$
\begin{exam}
We display here the matrices $M_3$ and $M_4$ of the restrictions of $\Psi$ to the homogeneous components $\Cal{T}_{\!\!3}^{pl}$ and $\Cal{T}_{\!\!4}^{pl}$ respectively:

$$
M_3= \begin{pmatrix}
1 & 1  \\
0 & 1
\end{pmatrix}\ ,
\quad\  
M_4= \begin{pmatrix}
1&1&1&1&1\\
0&1&0&1&1\\
0&0&1&0&1\\
0&0&0&1&2\\
0&0&0&0&1
\end{pmatrix}.
$$
\end{exam}

\begin{cor}
$(\Cal{T}^{pl},\searrow)$ is another description of the free magmatic algebra.
\end{cor}
Here is the explicit expression of $\Psi$ and $\Psi^{-1}$ on the following planar rooted trees:
$$\table$$

Now, we shall review the (unpublished) joint work of K. Ebrahimi-Fard and D. Manchon on finding a formula for the coefficient $c(\sigma,\tau)$ of tree $\sigma$ in $\Psi(\tau)$, for any trees $\sigma$ and $\tau$ in $T^{pl}$. Let $\sigma$ be any planar rooted tree, and $v, w$ be two vertices in the set $V(\sigma)$ of its vertices, define a partial order "$<$" as follows: $v<w$ if there is a path from the root to $w$ through $v$. The root is the minimal element, and leaves are the maximal elements. Define a refinement "$\ll$" of this order to be the transitive closure of the relation $R$ defined by: $vRw$ if $v<w$, or both $v$ and $w$ are linked to a third vertex $u \in V(\sigma)$, such that $v$ lies on the right of $w$, like this: $\arbrebbLab$. A further refinement "$\lll$" on $V(\sigma)$ is the total order recursively defined as follows: $v \lll w$ if and only if $v \lll w$ inside $V(\sigma_1)$ or $V(\sigma_2)$, or $v \in V(\sigma_2)$ and $w \in V(\sigma_1)$, where $\sigma = \sigma_1 \lbutcher\,\sigma_2$.\\ \\

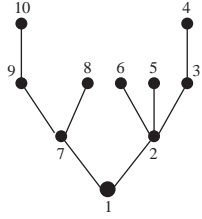
\begin{align*}
&\arbreebz
\end{align*}
\begin{align*}
& A\,planar\,rooted\,tree\,with\,its\,vertices\,labeled\,according\,to\,total\,order\,"\!\lll\!".
\end{align*}

\begin{thm}\label{thm of K. and M.}
For any planar rooted tree $\tau$ we have:
\begin{equation}
\Psi(\tau) = \sum_{\sigma \in T^{pl}}{c(\sigma,\tau)\sigma},
\end{equation}
where $c(\sigma,\tau)$ are nonnegative integers. An explicit expression for $c(\sigma,\tau)$ is given by the number of bijections $\psi:V(\sigma)\longrightarrow V(\tau)$ which are increasing from $(V(\sigma),\ll)$ to $(V(\tau),\lll)$, and such that $\psi^{-1}$ is increasing from $(V(\tau),<)$ to $(V(\sigma),<)$.
\end{thm}
\begin{proof}
This Theorem is proved using the induction on the degree $n$ of trees. The proof is trivial for $n=1, 2$. Given any planar rooted trees $\sigma, \tau \in T^{pl}_n$, such that $\tau$ can be written in a unique way as $\tau=\tau_1\lbutcher \tau_2$, we have:
\begin{equation}
c(\sigma,\tau_1\lbutcher \tau_2)=\sum_{v\in V(\sigma)}{c(\sigma^v,\tau_1)c(\sigma_v,\tau_2)},
\end{equation}
where $\sigma^v$ is the leftmost branch of $\sigma$ starting from $v$, and $\sigma_v$ is the corresponding trunk, i.e. what remains when the branch $\sigma ^ v$ is removed. This is immediate from the following computation:
\begin{align*}
\Psi(\tau)&=\Psi(\tau_1\lbutcher \tau_2)  \\
&= \Psi(\tau_1)\searrow \Psi(\tau_2)\\
&= \sum_{\sigma'\!\!,\,\sigma''\in T^{pl}}{c(\sigma',\tau_1)c(\sigma'',\tau_2)\,\sigma'\searrow \sigma''}.
\end{align*}
Denote by $b(\sigma,\tau)$ the number of bijections from $V(\sigma)$ onto $V(\tau)$ satisfying the growing conditions of Theorem \ref{thm of K. and M.}. Let $\psi$ be an increasing bijection from $(V(\sigma),\ll)$ to $(V(\tau),\lll)$, the decomposition $\tau=\tau_1\lbutcher \tau_2$ defines a partition of $V(\sigma)$ into two parts $V_i=\psi^{-1}(V(\tau_i)),\,i=1, 2$ such that $V_2\!\ll\!V_1$, which means that for any $v\in V_1$ and $w\in V_2$, either $w\!\ll\!v$ or they are incomparable. Such partitions are nothing but left admissible cuts \cite{HW08}. Put $\sigma_{V_1}$ and $\sigma_{V_2}$ to be the corresponding pruning and the trunk respectively.\\

As the inverse $\psi^{-1}$ moreover respects the order "$<$", there is a unique minimal element in $V_1$ for "$<$", namely $\psi^{-1}(v_1)$ where $v_1$ is the root of $\tau_1$. This means that the left cut considered here is also elementary, i.e. the pruning $\sigma_{V_1}$ is a tree. It is then clear that the restriction $\psi_i$ of $\psi$ to $\sigma_{V_i}$ is a bijection from $V(\sigma_{V_i})$ onto $V(\tau_i)$ which respects the growing conditions of the Theorem, for $i=1, 2$. Conversely, any vertex $v$ of $\sigma$ defines an elementary left cut by taking the leftmost branch $\sigma^v$ starting from $v$ and the corresponding trunk $\sigma_v$, and if $\psi':V(\sigma') \longrightarrow V(\tau_1)$ and $\psi'':V(\sigma'') \longrightarrow V(\tau_2)$ are two bijections satisfying the growing conditions of the Theorem, then the bijection $\psi:V(\sigma)\longrightarrow V(\tau)$ obtained from $\psi'$ and $\psi''$ also satisfies these conditions. Thus, we arrive that:
\begin{equation}\label{eq. of b}
b(\sigma,\tau_1\lbutcher \tau_2)=\sum_{v\in V(\sigma)}{b(\sigma^v,\tau_1)b(\sigma_v,\tau_2)},
\end{equation}
hence, the coefficients $c(-,-)$ and $b(-,-)$ satisfy the same recursive relations. This ends the proof of Theorem \ref{thm of K. and M.}. 
\end{proof}

\begin{exam}
$c(\arbreccc,\arbrecd)= 2$ according to the table above. Let us name the vertices as follows:
$$ \arbrecccLab\,,\,\,\,\,\,\,\,\,\ \arbrecdLab.$$

Take $\varphi\!:\!V(\arbreccc) \rightarrow V(\arbrecd)$ be a bijective map. \!\!We have $v_1\ll v_3,\ v_1\ll v_2\ll v_4,\ v_2\ll v_3$, as well as $w_1\lll w_2\lll w_3\lll w_4$. The growing conditions of Theorem \ref{thm of K. and M.} impose:
$$\varphi(v_1)=w_1,\ \varphi(v_2)\lll \varphi(v_3).$$
Hence we have:
\begin{align*}
\varphi(v_1)= w_1 & \hskip 10mm\varphi(v_1)= w_1\\
\varphi(v_2)= w_2 & \hskip 10mm\varphi(v_2)= w_2\\
\varphi(v_3)= w_3 & \hskip 2mm\hbox{ or } \hskip 2mm\varphi(v_3)= w_4\\
\varphi(v_4)= w_4 & \hskip 10mm\varphi(v_4)= w_3
\end{align*}
The inverse of both bijections obviously respect the order $"<"$. Hence we find two bijections verifying the growing conditions of Theorem \ref{thm of K. and M.}, thus recovering $b(\arbreccc, \arbrecd)= 2$.
 
\end{exam}

\subsection{From planar to non-planar rooted trees}
Corresponding to the coefficients $c(\sigma,\tau)$, with their explicit expressions, in the matrix of the restriction of the linear map $\Psi$ to any homogeneous component $\Cal{T}^{pl}_n\!,$ we try to find a similar expression in the non-planar case: in other words, we build up and explicitly describe the map $\widetilde{\Psi}_S$ in the diagram of the introduction.\\

\begin{defn}
The grafting product $"\to"$ is a bilinear map defined on the vector space $\Cal{T}$such that:
\begin{equation}
s \to t =\sum_{v\,\in V(t)}{ s \to_v t},
\end{equation}
for any $s, t\!\in\!\Cal{T}\!$, where "$s \to_v t$" is the (non-planar) rooted tree obtained by grafting the tree $s$ on the vertex $v$ of the tree $t$.  
\end{defn}

\begin{exam}
$$\racine\to\arbrea=\arbreba+\arbrebb,\hskip 8mm \arbrea\to\arbrea=\arbreca+\arbrecc.$$
\end{exam}

The space $\Cal{T}\!$, with this type of grafting, forms a special structure of algebra called pre-Lie algebra, as it will be shown later in Section \ref{free pre-lie algebra}. Recall that the symmetry factor of any (non-planar) rooted tree $s$ is the number $sym(s)$ of all automorphisms $\Theta\!:\!V(s)\!\longrightarrow\!V(s)$ which are increasing from $(V(s),<)$ onto $(V(s),<)$. This definition is equivalent to the recursive definition in \cite{C.B04}.\\

Let $\overline{\Psi}= \pi \circ \Psi$ be the linear  map from $\Cal{T}^{pl}$ onto $\Cal{T}$, where $\pi$ is the "forget planarity" projection. 
\diagrama{
\xymatrix{
\Cal {T}^{pl}\ar[r]^{\Psi}\ar[dr]|{\overline{\Psi}}
&\Cal T^{pl}\ar@{->>}[d]^{\pi}\\
& \Cal T}}
Obviously, $\overline{\Psi}$ is a morphism of algebras from $(\Cal{T}^{pl},\lbutcher)$ into $(\Cal{T},\to)$. One of the important results obtained in this article is the following:

\begin{thm}\label{main}
Let $\tau$ be any planar rooted tree, we have:
\begin{equation}
\overline{\Psi}(\tau)=\sum_{s \in T}{\alpha(s,\tau)s},
\end{equation}
where $\alpha(s,\tau)$ are nonnegative integers. The coefficients $\alpha(s,\tau)$ coincide with the numbers $\overline{b}(s,\tau) = \tilde{b}(s,\tau)/sym(s) $, where $sym(s)$ is the symmetry factor of $s$ described above and $\tilde{b}(s,\tau)$ is the number of bijections $ \varphi : V(s)\longrightarrow V(\tau) $ which are increasing from $(V(s),<)$ into $(V(\tau),\lll)$, such that $\varphi^{-1}$ is increasing from $(V(\tau),<)$ into $(V(s),<)$. 
\end{thm}

\begin{proof}
Note that the restriction of $\overline{\Psi}$ to any homogeneous component $\Cal{T}^{pl}_ n $ reduces the square matrix of the coefficients $c(\sigma,\tau)$ to a rectangular matrix $[\![\alpha(s,\tau)]\!]_{\scriptstyle s \in T_n ,\, \tau \in T^{pl}_n}$. For any planar rooted tree $\tau$, we have:
\begin{equation}\label{alpha}
\alpha(s,\tau)=\sum_ {{\scriptstyle \sigma \in T^{pl}} \atop {\scriptstyle \pi(\sigma)=s}}{c(\sigma,\tau)},
\end{equation}                            
where $s$ is a (non-planar) rooted tree. We prove Theorem \ref{main} using the induction on the degree of trees. The proof is trivial in the cases $n=1,2$. Let $\tau \in T^{pl} _n$, with $\tau = \tau_1 \lbutcher \tau_2$, then:
\begin{equation}\label{alphabutcher}
\alpha(s,\tau_1 \lbutcher \tau_2)=\!\!\!\!\!\!\!\!\sum_{\genfrac{}{}{0pt}{}{\scriptstyle \sigma \in T^{pl}} {\scriptstyle \pi(\sigma)=s ,\, v \in V(\sigma)}}\!\!\!\!\!\!\!\!{c(\sigma^v,\tau_1)c(\sigma_v,\tau_2)},
\end{equation}
which is immediate from \eqref{alpha}, where $\sigma^v$ is the leftmost branch of $\sigma$ starting from $v$, and $\sigma_v$ is the corresponding trunk.\\

Now, let $s$ be any (non-planar) rooted tree in $T_n$ and $ \varphi : V(s)\rightarrow V(\tau) $ be a bijection which satisfies the growing conditions given in Theorem \ref{main}. Then we can define from these conditions a structure of poset on the set $V(s)$ of vertices of $s$ , as follows: for $v,w \in V(s)$, $vRw$ if and only if $v<w$ or there is $u \in V(s)$ such that each of $v$ and $w$ are related with $u$ by an edge, and $\varphi(v) \lll \varphi(w)$. We denote by $\ll_{\varphi}$ the transitive closure of the relation $R$.\\

This structure determines a planar rooted tree $\sigma$ such that $\pi(\sigma)=s$, with the associated partial order $\ll$ on the set $V(\sigma)$ of vertices of $\sigma$, together with a poset isomorphism $\vartheta:(V(\sigma),\ll) \rightarrow (V(s),\ll_{\varphi})$, which in turn defines a bijection $\varphi^!:=\varphi \circ \vartheta:V(\sigma)\rightarrow V(\tau)$, which is increasing from $(V(\sigma),\ll)$ into $(V(\tau),\lll)$, such that ${\varphi^{!}}^{-1}$ is increasing from $(V(\tau),<)$ into $(V(\sigma),<)$. The planar rooted tree $\sigma$ is unchanged if we replace $\varphi$ by $\varphi \circ \vartheta'$ with $\vartheta'\!\!\in\!Aut(s)$. Moreover, for any $\varphi, \psi:V(s) \rightarrow V(\tau)$ satisfying the growing conditions of Theorem \ref{main}, we have: 
\begin{equation}\label{Aut}
\varphi^! = \psi^! \Leftrightarrow \varphi = \psi \circ \gamma, \hbox{ for } \gamma \in Aut(s). 
\end{equation}
Indeed, if $\varphi^! = \psi^!$, then $\gamma:= \psi^{-1} \circ \varphi: V(s) \rightarrow V(s)$ is a bijection which respects the partial order $"<"$, hence an element of $Aut(s)$, such that $\varphi = \psi \circ \gamma$. The inverse implication is obvious.\\

Let $\widetilde{B}(s, \tau)$ (respectively $B(\sigma, \tau)$) be the set of bijections $\varphi:V(s) \rightarrow V(\tau)$ (respectively $\psi:V(\sigma) \rightarrow V(\tau)$) verifying the growing conditions of Theorem \ref{main} (respectively Theorem \ref{thm of K. and M.}), and suppose $\pi(\sigma)=s$. Denote by $\tilde{b}(s, \tau)$ (respectively $b(\sigma, \tau)$) the cardinal number of $\widetilde{B}(s, \tau)$ (respectively $B(\sigma, \tau)$). Now, define : 

\begin{equation}
\overline{b}(s,\tau):= \!\!\!\!\!\!\!\!\sum_{\scriptstyle \sigma \in T^{pl}\!,\,\pi(\sigma)=s}{\!\!\!\!\!\!\!\!b(\sigma,\tau)}.
\end{equation}

Then, according to \eqref{Aut}, we have: 
$$\overline{b}(s,\tau)=\widetilde{b}(s,\tau)/sym(s).$$
We also have for $\tau=\tau_1 \lbutcher \tau_2$: 
\begin{equation}\label{overline b butcher}
\overline{b}(s,\tau_1 \lbutcher \tau_2)= \!\!\!\!\!\!\!\!\sum_{{\scriptstyle \sigma \in T^{pl}} \atop {\scriptstyle \pi(\sigma)=s,\, v \in V(\sigma)}}{\!\!\!\!\!\!\!\!b(\sigma^v,\tau_1)b(\sigma_v,\tau_2)}.
\end{equation}

The coefficients $c(-,-)$ and $b(-,-)$ coincide by Theorem \ref{thm of K. and M.}. So, from \eqref{alphabutcher} and \eqref{overline b butcher}, the coefficients $\alpha(-,-)$ and $\overline{b}(-,-)$ satisfy the same recursive relations, which proves the Theorem.        
\end{proof}
\begin{exam}
$\alpha(\arbrebb, \arbrebb)= 1$ in the formula of $\,\overline{\Psi}(\arbrebb)$. Name the vertices as follows:
$$ \arbrebbLabel\,,\,\,\,\,\,\,\,\,\ \arbrebbLabe.$$

Let $\psi\!:\!V(\arbrebb) \rightarrow V(\arbrebb)$ be a bijective map. \!\! We have $v_1 < v_2,\ v_1 < v_3$, as well as $w_1 \lll w_2 \lll w_3$. The growing conditions of Theorem \ref{main} impose $\psi(v_1)=w_1$. Hence we have:
\begin{align*}
\psi(v_1)= w_1  & \hskip 10mm\psi(v_1)= w_1 \\
\psi(v_2)= w_2  & \hskip 2mm\hbox{ or } \hskip 2mm\psi(v_2)= w_3 \\
\psi(v_3)= w_3  & \hskip 10mm\psi(v_3)= w_2
\end{align*}
The inverse of these bijections obviously respect the order $"<"$. Hence we find two bijections verifying the growing conditions of Theorem \ref{main}, thus $\tilde{b}(\arbrebb, \arbrebb)= 2$, but $sym(\arbrebb)=2$, then we obtain $\overline{b}(\arbrebb, \arbrebb)=1$. 
\end{exam}

We want to describe a family of linear isomorphisms $\widetilde{\Psi}:\Cal T \longrightarrow \Cal T\!,$ which make the following diagram commute:
\diagrama{
\xymatrix{\Cal {T}^{pl}\ar[r]^{\Psi}\ar[dr]|{\overline{\Psi}}\ar@{->>}[d]_{\pi}
&\Cal T^{pl}\ar@{->>}[d]^{\pi}\\
\Cal T\ar@{.>}[r]_{\widetilde\Psi} & \Cal T
}}
For any (non-planar) rooted tree $t$, choose $\sigma=S(s)$ to be a planar rooted tree with $\pi(\sigma)=s$. This defines a section $S:\Cal T \longrightarrow \Cal T^{pl}$ of the projection $\pi$, (i.e.) $\pi \circ S =Id_{\Cal T}$. One can note that the map $S$ is not unique, for example, if $n=4$, we have:\\
$T_4=\{\arbreca, \!\arbrecb\!, \arbrecc\!, \arbrecd\}$ and $T^{pl}_4= \{\arbreca, \!\arbrecb\!, \arbrecc\!, \arbreccc\!, \arbrecd\}$, then we can define $S$ as:\\
$S(\arbreca)= \arbreca,\,S(\!\arbrecb\!)=\!\arbrecb\!,\,S(\arbrecd)=\!\arbrecd\!,$ and one can choose for $S(\arbrecc)$ between $\arbrecc$ and $\arbreccc\!.$
\\

Let $S$ be a section of $\pi$. Define $\widetilde{\Psi}_S:=\overline{\Psi} \circ S$ to be the linear map from $\Cal T$ into $\Cal T\!,$ which makes the following diagram commute:
\diagrama{
\xymatrix{
\Cal {T}^{pl}\ar[r]^{\Psi}\ar[dr]|{\overline{\Psi}}\ar@{->>}[d]^{\pi}
&\Cal T^{pl}\ar@{->>}[d]^{\pi}\\
\Cal T \ar@{.>}[r]_{{\widetilde\Psi}_S}\ar@<6pt>@{.>}[u]^S & \Cal T}} 
 
\begin{cor}\label{cor 5}
For any (non-planar) rooted tree $t$, we have:
\begin{equation}
\widetilde{\Psi}_S(t) = \sum_{s \in T}{\beta_S(s, t)s},
\end{equation}
where $\beta_S(s, t)$ are nonnegative integers. The coefficients $\beta_S(s, t)$, which depend on the section map $S\!,$ can be expressed by the number $\overline{b}(s,\tau) = \widetilde{b}(s,\tau)/sym(s)$ described in Theorem \ref{main}, with $\tau=S(t)$.
\end{cor}

\begin{proof}
Note that the restriction of $\,\widetilde{\Psi}_S$ to any homogeneous component $\Cal{T}_n$ reduces the matrix of the coefficients $\alpha(s,\tau)$ to a upper triangular unipotent matrix $[\![\beta_S(s, t)]\!]_{\scriptstyle\,s,\,t \in T_n}$. Let $t$ be any (non-planar) rooted tree, and let us choose the section map $S$ such that $S(t)=\tau$ is a planar rooted tree, then:
$$\widetilde{\Psi}_S(t) = \overline{\Psi}(\tau) = \sum_{s \in T}{\alpha(s,\tau)s},$$
which means that the coefficients $\beta_S(s, t)$ and $\alpha(s,\tau)$ are the same. Hence, it can be expressed by the number $\overline{b}(s,\tau)$ in the same way than the coefficients $\alpha(s,\tau)$. From Theorem \ref{main}, we found that the restriction of $\overline{\Psi}$ to any homogeneous component $\Cal T^{pl}_n$ reduces the matrix of the coefficients $c(\sigma,\tau)$ to a rectangular matrix. Now, the restriction of $\widetilde{\Psi}_S$ to any homogeneous component $\Cal T_n$ can be represented by the restriction of $\overline{\Psi}$ on the component $S(\Cal T_n)$ (this representation depends on the section map $S$), which means that the matrix of $\beta_S(s, t)$'s is a upper triangular unipotent matrix, because we have:
$$\widetilde{\Psi}_S(t) =t+\hbox{ terms of higher energy}.$$
\end{proof}

\section{Free Pre-Lie algebras}\label{free pre-lie algebra}

The concept of "Pre-Lie algebras" appeared in many works under various names. The first appearance of this notion can be traced back to 1857 in a paper by A. Cayley \cite{A.C57}. In 1961, J. L. Koszul studied this type of algebras in his work \cite{J.K61}. E. B. Vinberg and M. Gerstenhaber in 1963 independently presented  the concept under two different names; "right symmetric algebras" and "pre-Lie algebras" respectively \cite{E.V63, M.G63}. Other denominations, e.g. "Vinberg algebras", appeared since then. "Chronological algebras" is the term used by A. Agrachev and R. V. Gamkrelidze in their work on \textsl{nonstationary vector fields} \cite{AR81}. The term "pre-Lie algebras" is now the standard terminology. We shall now review some basics and topics related to pre-Lie algebras.

\begin{defn}

Let $\Cal{A}$ be a vector space over a field $K$ together with a bilinear operation ``$\vartriangleright$''. Then $\Cal{A}$ is said to be a left pre-Lie algebra, if the map $\vartriangleright$ satisfies the following identity:
\begin{equation}\label{pre-lie identity}
(x \vartriangleright y)\vartriangleright z - x \vartriangleright  (y \vartriangleright z) = (y \vartriangleright x)\vartriangleright z - y \vartriangleright  (x \vartriangleright z), \forall x,y,z \in \Cal{A}.
\end{equation}   
\end{defn}
The identity \eqref{pre-lie identity} is called the left pre-Lie identity, and it can be written as:
\begin{equation}
L_{[x,y]} = [L_x,L_y], \forall x,y \in \Cal{A},
\end{equation}      
where for every element $x$ in $\Cal A$, there is a linear transformation $L_x$ of the vector space $\Cal A$ defined by $L_x(y)={x}\vartriangleright{y}, \forall y\!\in\!\Cal A$, and $[x,y]={x}\vartriangleright{y}-{y}\vartriangleright{x}$ is the commutator of the elements $x$ and $y$ in $\Cal A$. The usual commutator $[L_x,L_y]=L_xL_y-L_yL_x$ of the linear transformations of $\Cal A$ defines a structure of Lie algebra over $K$ on the vector space $L(\Cal A)$ of all linear transformations of $\Cal A$. For any pre-Lie algebra $\Cal{A}$, the bracket $[-,-]$ satisfies the Jacobi identity, hence induces a structure of Lie algebra on $\Cal{A}$.\\

In the vector space  $\Cal T$ spanned by the rooted trees, the grafting operation $"\to"$ satisfies the pre-Lie identity, since for any $s, t, t'\!\in T\!,$ we have:
\begin{align*}
s \rightarrow ( t \rightarrow t') - (s \rightarrow t) \rightarrow t' &= s \rightarrow(\sum_{v\in V(t')}{t \rightarrow_{v}t'}) - (\sum_{u \in V(t)}{s \rightarrow_{u}t}) \rightarrow t'\\
&= \sum_{v \in V(t')}{s \rightarrow(t \rightarrow_{v}t')} - \sum_{u \in V(t)}{(s \rightarrow_{u}t)} \rightarrow t'\\
&= \sum_{v \in V(t')}{\sum_{v'\in V(t'')}{s \rightarrow_{v'}(t \rightarrow_{v}t')}} \\
& \hskip 5mm - \sum_{v \in V(t')}{\sum_{u \in V(t)}{(s \rightarrow_{u}t)} \rightarrow_{v}t'}, \,\,\, [t'' = t \rightarrow_{v}t']\\ 
&= \sum_{v \in V(t')}{\sum_{v' \in V(t')}{s \rightarrow_{v'}(t \rightarrow_{v}t')}},
\end{align*}

Obviously symmetric in $s$ and $t$.\\

Free pre-Lie algebras have been handled in terms of rooted trees by F. Chapoton and M. Livernet \cite{CL01}, who also described the pre-Lie operad explicitly, and by A. Dzhumadil'daev and C. L\"ofwall independently \cite{AC02}. For an elementary version of the approach by Chapoton and Livernet without introducing operads, see e.g. \cite[Paragraph 6.2]{D.M}:

\begin{thm}\label{d-generators}
Let $k$ be a positive integer. The free pre-Lie algebra with $k$ generators is the vector space $\Cal T$ of (non-planar) rooted trees with $k$ colors, endowed with grafting.
\end{thm}

\subsection{Construction of monomial bases}

A. Agrachev and R. V. Gamkrelidze, in their work \textsl{Chronological algebras and nonstationary vector fields} \cite{AR81}, described a pre-Lie algebra isomorphism between the free pre-Lie algebra generated by a (non-empty) set and the tensor product of the universal enveloping algebra of the underlying Lie algebra with the linear span of the generating set. This pre-Lie algebra isomorphism will be the focus of our attention in this section. Using this isomorphism, we shall review the construction by A. Agrachev and R. Gamkrelidze of monomial bases in free pre-Lie algebras.\\

We described the free pre-Lie algebra in terms of rooted trees. We give below its definition in terms of a universal property.
  
\begin{defn}\label{free pre-Lie algebra}
Let $\Cal A$ be a pre-Lie algebra and $E$ a (non-empty) set with an embedding map $i:E \hookrightarrow \Cal A $. Then $\Cal A$ is said to be a free pre-Lie algebra generated by $E$, if for any pre-Lie algebra $ \Cal B $ and map $ f_\circ:E \longrightarrow \Cal B $ there is a unique pre-Lie algebra homomorphism $ f:\Cal A \longrightarrow \Cal B $, which makes the following diagram commute:
\diagrama{
\xymatrix{ E \ar@{^{(}->}[r]^-{i}\ar[d]_-{f_\circ}
& \Cal A \ar[dl]|-{f} \\
\Cal B
}}
\end{defn}

The free pre-Lie algebra generated by a set $E$ is unique up to isomorphism. It can be obtained as the quotient of the free magmatic algebra $A_E$ with generating set $E$ by the two-sided ideal generated by elements on the form:
\begin{equation}\label{ideal}
x\vartriangleright(y\vartriangleright z)-y\vartriangleright(x\vartriangleright z)-(x\vartriangleright y-y\vartriangleright x)\vartriangleright z,\,for \,x,y,z \in A_E. 
\end{equation}

From Definition \ref{free pre-Lie algebra}, we have that any pre-Lie algebra $\Cal B$ generated by a subset $E\subset\Cal{B}$ is isomorphic to a quotient of the free pre-Lie algebra $\Cal A$ on $E$ by some ideal. Indeed, from the freeness universal property of $\Cal A$, there is a unique homomorphism $f$, which is surjective. The quotient of $\Cal A$ by the kernel of $f$ is isomorphic to $\Cal B$, as in the following commutative diagram:
     
\diagrama{
\xymatrix{
E \ar@{^{(}->}[r]^-{i}\ar@{^{(}->}[dr]|{j} & \Cal A \ar@{->>}[r]^-{q} \ar@{->>}[d]^-{f} & \Cal{A}/kerf \ar[dl]|{\tilde{=}}\\
& \Cal B 
}}
\noindent where $q$ is the quotient map.\\

We shall denote by $A_E$ the free magmatic algebra, and by $\Cal{A}_E$ the free pre-Lie algebra generated by the (non-empty) set $E$. The algebra $A_E$ has a natural grading, where the elements of degree 1 are linear combinations of the elements of $E$. The algebra $\Cal{A}_E$ can be defined as the quotient of $A_E$ by the ideal \eqref{ideal}. This induces a grading on $\Cal{A}_E$, in which the elements of degree 1 are again the linear combinations of the elements of $E$, by identifying the set $E$ with its image under the factorization.\\

Denote by $[\Cal{A}_E]$ the underlying Lie algebra of $\Cal{A}_E$, and $\Cal{U}[\Cal{A}_E]$ its universal enveloping algebra\footnote{\,For more details about the universal enveloping algebra see \cite{J.D74, Ch.R93}.}. The structure of algebra defined on $\Cal{U}[\Cal{A}_E]$ is endowed with the grading deduced from the grading of $\Cal{A}_E$.\\

The representation of the Lie algebra $[\Cal{A}_E]$ by the linear transformations $(x \mapsto L_x$, for $x \in\Cal{A}_E)$ of the algebra $\Cal{A}_E$ is uniquely extended to a representation by the linear transformations $(u \mapsto L_u$, for $u \in\Cal{U}[\Cal{A}_E])$ of the enveloping algebra $\Cal{U}[\Cal{A}_E]$, which makes the following diagram commute:

\diagrama{
\xymatrix{
\Cal{A}_E \ar@{^{(}->}[rr]^-{i} \ar[dr]_{L} && T(\Cal{A}_E) \ar@{->>}[rr]^-{q} \ar[dl]^-{L^{\prime}} && \Cal{U}[\Cal{A}_E] \ar[dl]^-{L\,\,\,\,\,\,\,\,\,\,\,\,\,\,\,\,,}\\
& End(\Cal{A}_E) \ar[rr]^-{\tilde{=}} && End(\Cal{A}_E) 
}}

\noindent where $T(\Cal{A}_E)$ is the tensor algebra of $\Cal{A}_E$, and $L'$ is the linear extension of $L$ that is induced by the universal property of the tensor algebra.  

\begin{lem}\label{linear span}
The linear span of the set 
$$ L_{\Cal{U}[\Cal{A}_E]}^{}E = {\{L_{u}s \,|\, u \in \Cal{U}[\Cal{A}_E], s \in E }\} \subset \Cal{A}_E $$
coincides with the entire algebra $\Cal{A}_E$. 
\end{lem}
\begin{proof}
See \cite[Lemma 1.1]{AR81}.
\end{proof}

Define $\Cal{B}_E=\Cal{U}[\Cal{A}_E] \otimes \overline{E}$ to be the tensor product of the vector space $\Cal{U}[\Cal{A}_E]$ with the linear span $\overline{E}$ of the set $E$. The space $\Cal{B}_E$ has a structure of algebra over $K$ with the following multiplication:
\begin{equation}\label{multiplication}
 (u_1 \otimes s_1)(u_2 \otimes s_2)=((L_{u_1}s_1) \circ u_2) \otimes s_2, \, \forall u_1,u_2 \in \Cal{U}[\Cal{A}_E], s_1, s_2 \in \overline{E},
\end{equation}
where ``$\circ$'' is the bilinear associative product in $\Cal{U}[\Cal{A}_E]$.\\

The grading of the algebra $\Cal{U}[\Cal{A}_E]$ uniquely determines a grading of $\Cal{B}_E$, by setting the degree of the element $u \otimes s$ equal to the degree of $u$ plus 1. One can verify that the multiplication defined in \eqref{multiplication} satisfies the pre-Lie identity, which means that $\Cal{B}_E$ is a graded pre-Lie algebra.
\begin{thm}\label{main ch. thm}
The graded pre-Lie algebra $\Cal{B}_E$ is isomorphic to the free pre-Lie algebra $(\Cal{A}_E,\rhd)$.
\end{thm}
\begin{proof}
Let $f_{\circ}:E \longrightarrow \Cal{B}_E$ be a map defined by $f_{\circ}(s)=1 \otimes s\,, \forall s \in E$, where $1$ is the unit element of $\Cal{U}[\Cal{A}_E]$. Using the freeness property of the pre-Lie algebra $\Cal{A}_E$, there is a unique homomorphism $f:\Cal{A}_E \longrightarrow \Cal{B}_E$, such that:
$$f(s)= f_{\circ}(s)= 1 \otimes s, \forall s \in {E}\subset \Cal{A}_E.$$

From Lemma \ref{linear span}, we have that for any element $x$ in $\Cal{A}_E$ there exists $u \in \Cal{U}[\Cal{A}_E]$ and $s \in E$, such that $x = L_{u}s$. Now, define $f$ as:
\begin{equation}\label{f}
f(L_{u}s)=u \otimes s, \forall x=L_{u}s \in \Cal{A}_E.
\end{equation}

Then the map $f$ with \eqref{f} is bijective (see \cite[Theorem 1.1]{AR81}), hence it is an isomorphism, which proves the Theorem. 
\end{proof}
Choose a total order on the elements of $E$. Then as a corollary of Theorem \ref{main ch. thm} and the Poincar\'e-Birkhoff-Witt Theorem, we obtain that: 
\begin{equation}\label{formula A}
\Cal{A}_n\tilde{=}\Cal{B}_n=\Cal{U}_{n-1}\otimes \overline{E}, \forall n\geq\!1,
\end{equation}
where, for any $n\geq 2$, a basis of $\Cal{U}_{n-1}$ is given by: 
$$\left \{ x^{e_1}_{\!j_{1}} \circ \cdots \circ x^{e_r}_{\!j_{r}}\ \big|\ \sum^{r}_{k=1}{j_k}=n-1, \hbox{ and } x^{e_1}_{\!j_{1}} \geq \cdots \geq x^{e_r}_{\!j_{r}}\right \}.$$ 
Here we use a monomial basis ${x_j^{1}, \ldots, x_j^{d_j}}$ of the subspace $\Cal{A}_j$, for any $j= 1, \ldots, n-1$, given by the induction hypothesis. We endow  this basis with the total order $x^{1}_{\!j} < \ldots < x^{d_j}_{\!j} $, which in turn defines a total order on the basis of $\Cal{A}_1 \oplus \ldots \oplus \Cal{A}_{n-1}$, obtained by the disjoint union, by asking that $x_j^{r} > x_{j'}^{r'}$ if $j>j'$.\\
 
Hence, using formula \eqref{formula A} and the isomorphism $f$ described in \eqref{f} , we get the following monomial basis for the homogeneous component $\Cal{A}_n$:
\begin{equation}\label{equation B}
\left\{ x^{e_1}_{\!j_{1}} \rhd \big (x^{e_2}_{\!j_{2}} \rhd (\cdots \rhd (x^{e_r}_{\!j_{r}}\rhd s_j) \cdots)\big )\ \big|\ \sum^{r}_{k=1}{j_k}=n-1,\ x^{e_1}_{\!j_{1}} \geq \cdots \geq x^{e_r}_{\!j_{r}} \hbox{ and } s_j \in E \right\}.
\end{equation}
 
\begin{align*}
\Cal{A}_1 &\ \tilde{=}\ \Cal{U}_0\otimes\overline{E}\\
& =\ \prec 1\otimes s\,|\,1\in K, s \in E\succ,\\
\Rightarrow\Cal{A}_1 &=\ \prec L_1s=s\,|\,s\in E\succ=\overline{E}.\\\\
\Cal{A}_2 &\ \tilde{=}\ \Cal{U}_1\otimes\Cal{A}_1\\
& =\ \prec s_1\otimes s_2\,|\, s_1, s_2 \in E\succ,\\  
\Rightarrow\Cal{A}_2 &=\ \prec L_{s_1}s_2=s_1\vartriangleright s_2\,|\,s_1, s_2 \in E\succ\!\!.\\\\
\Cal{A}_3 &\ \tilde{=}\ \Cal{U}_2\otimes\Cal{A}_1\\
& =\ \prec x_2^e\otimes s_{j_1}, (x_1^{e_1}\circ x_1^{e_2})\otimes s_{j_2}\ |\ e = 1, \ldots, d_2,\ e_1, e_2= 1, \ldots, d, e_1 \geq e_2, s_{j_1}, s_{j_2} \in E\succ,\\
&\hskip -13mm \Rightarrow\hbox{ A monomial basis of $\Cal{A}_3$ is then given by: }\\
&\hskip  5mm\left\{ (s_1 \rhd s_2) \rhd s_3 |\ s_1, s_2, s_3 \in E \right\} \sqcup \left\{ s_1 \rhd (s_2 \rhd s_3) |\ s_1, s_2, s_3 \in E,\ s_1 \geq s_2 \right\}.
\\\\
\Cal{A}_4 &\ \tilde{=} \Cal{U}_3\otimes\Cal{A}_1\\
& = \prec x_3^e\otimes s_{j_1}, (x_2^{e'}\circ x_1^{e_1})\otimes s_{j_2}, (x_1^{e_2} \circ x_1^{e_3} \circ x_1^{e_4})\otimes s_{j_3}\ |\ e = 1, \ldots, d_3, e' = 1, \ldots, d_2,\\
& \hskip 5mm\ e_1, e_2, e_3, e_4= 1, \ldots, d_1, e_2 \geq e_3\geq e_4, s_{j_1}, s_{j_2}, s_{j_3} \in E\succ,
\end{align*}
\begin{align*}
& \hskip -13mm \Rightarrow\hbox{ A monomial basis of $\Cal{A}_4$ is then given by: }\\
& \hskip -13mm\left\{\left((s_1 \rhd s_2) \rhd s_3 \right) \rhd s_4 |\ s_j \in E,\hbox{ for } j=1, 2, 3, 4\right\} \sqcup \left\{\left(s_1 \rhd (s_2 \rhd s_3)\right)\rhd s_4 |\ s_j \in E,\hbox{ for } j=1, 2, 3, 4,\ s_1 \geq s_2 \right\} \sqcup\\
& \hskip -13mm\left\{(s_1 \rhd s_2) \rhd (s_3 \rhd s_4) |\ s_j \in E,\hbox{ for } j=1, 2, 3, 4\right\} \sqcup \left\{s_1 \rhd (s_2 \rhd (s_3 \rhd s_4)) |\ s_j \in E,\hbox{ for } j=1, 2, 3, 4, s_1 \geq s_2 \geq s_3 \right\}.\\\\
\Cal{A}_5 &\ \tilde{=} \Cal{U}_4\otimes\Cal{A}_1\\
& = \prec x_4^e\otimes s_{j_1}, (x_3^{e'}\circ x_1^{e_1})\otimes s_{j_2}, (x_2^{e'_2}\circ x_2^{e''_2})\otimes s_{j_3}, (x_2^{e'''_2} \circ x_1^{e_2} \circ x_1^{e_3})\otimes s_{j_4}, (x_1^{e_4} \circ x_1^{e_5} \circ x_1^{e_6}\circ x_1^{e_7})\otimes s_{j_5} \ |\\
& \hskip 5mm e = 1, \ldots, d_4, e' = 1, \ldots, d_3, e'_2, e''_2, e'''_2= 1, \ldots, d_2, \ e_i= 1, \ldots, d_1, \forall i=1, \ldots, 7, e'_2 \geq e''_2, e_2 \geq e_3,\\
& \hskip 5mm e_4 \geq e_5\geq e_6\geq e_7, s_{j_r} \in E, \forall r=1, \ldots, 5\succ,\\
&\hskip -13mm \Rightarrow\hbox{ A monomial basis of $\Cal{A}_5$ is then given by: }\\
&\hskip -13mm \left\{((s_1 \rhd s_2) \rhd s_3)\rhd s_4) \rhd s_5 |\ s_j \in E,\hbox{ for } j=1,\ldots, 5 \right\}\sqcup\Big\{(s_1 \rhd (s_2 \rhd s_3))\rhd s_4) \rhd s_5 |\ s_j \in E,\hbox{ for } j=1, \ldots, 5,\\ 
&\hskip -12mm s_1 \geq s_2 \Big\}\sqcup\left\{((s_1 \rhd s_2) \rhd (s_3 \rhd s_4)) \rhd s_5 |\ s_j \in E,\hbox{ for } j=1, \ldots, 5 \right\}\sqcup\Big\{(s_1 \rhd (s_2 \rhd (s_3 \rhd s_4))) \rhd s_5 |\ s_j \in E,\hbox{ for } \\
&\hskip -12mm j=1, \ldots, 5, s_1 \geq s_2 \geq s_3 \Big\}\sqcup\left\{((s_1 \rhd s_2) \rhd s_3) \rhd (s_4 \rhd s_5) |\ s_j \in E,\hbox{ for } j=1, \ldots, 5 \right\}\sqcup\Big\{( s_1 \rhd (s_2 \rhd s_3))\rhd (s_4\rhd\\ 
&\hskip -12mm s_5)\ |\ s_j \in E, \hbox{ for } j=1, \ldots, 5, s_1 \geq s_2 \Big\}\sqcup\Big\{(s_1 \rhd s_2) \rhd ( (s_3 \rhd s_4)\rhd s_5) |\ s_j \in E,\hbox{ for } j=1, \ldots, 5, s_1 \rhd s_2 \geq s_3 \rhd s_4 \Big\} \\
&\hskip -12mm \sqcup\Big\{(s_1 \rhd s_2)\rhd (s_3 \rhd (s_4\rhd s_5)) |\ s_j \in E,\hbox{ for } j=1, \ldots, 5, s_3 \geq s_4\Big\}\sqcup\Big\{s_1 \rhd (s_2 \rhd (s_3 \rhd(s_4\rhd s_5)))\ |\ s_j \in E, \hbox{ for }\\
&\hskip -12mm j=1, \ldots, 5, s_1 \geq s_2\geq s_3 \geq s_4 \Big\}.\\
&\vdots\\
& etc.
\end{align*}
  
\subsection{Base change from any monomial basis to the rooted tree basis} 
 
We relate now any Agrachev-Gamkrelidze type monomial basis in a free pre-Lie algebra, obtained from the formula \eqref{formula A}, with the presentation of the free pre-Lie algebra as the linear span $\Cal{T}$ of the (non-planar) rooted trees with one generator $\{\racine\,\}$, endowed with the grafting $"\to"$. In the following, we give the tree expansions of the first five homogeneous components of such a monomial basis:
\begin{align*}
\Cal{T}_1&=\prec e_1=\racine\,\succ\!\!.\\  
\Cal{T}_2&=\prec\racine\to\!\racine\,\succ\,\,\,=\,\,\,\prec e_1=\arbrea\,\succ\!\!.\\ 
\Cal{T}_3&=\prec(\racine\to\!\racine)\to\!\racine\,, \racine\to\!(\racine\to\!\racine)\,\succ\,\,\,=\,\,\,\prec e_1=\arbreba, e_2=\arbreba\!+\!\arbrebb\,\succ\!\!.\\ 
\Cal{T}_4&=\prec((\racine\to\!\racine)\to\!\racine)\to\!\racine, (\racine\to\!(\racine\to\!\racine))\to\!\racine, (\racine\to\!\racine)\to\!(\racine\to\!\racine), \racine\to\!(\racine\to\!(\racine\to\!\racine))\,\succ\\ 
&=\prec e_1=\arbreca, e_2=\arbreca+\arbrecb, e_3=\arbreca+\arbrecc, e_4=\arbreca+\arbrecb+3 \arbrecc+\arbrecd\succ\!\!.
\end{align*}
\begin{align*}
\Cal{T}_5&=\prec(((\racine\to\!\racine)\to\!\racine)\to\!\racine)\to\!\racine, ((\racine\to\!(\racine\to\!\racine))\to\!\racine)\to\!\racine, ((\racine\to\!\racine)\to\!(\racine\to\!\racine))\to\!\racine, (\racine\to\!(\racine\to\!\\ 
&\hskip 5mm (\racine\to\!\racine)))\to\!\racine, (((\racine\to\!\racine)\to\!\racine)\to\!(\racine\to\!\racine), ((\racine\to\!(\racine\to\!\racine))\to\!(\racine\to\!\racine), (\racine\to\!\racine)\to\!((\racine\to\!\racine)\to\!\racine),\\ 
&\hskip 5mm (\racine\to\!\racine)\to\!(\racine\to\!(\racine\to\!\racine)), \racine\to\!(\racine\to\!(\racine\to\!(\racine\to\!\racine)))\succ\!\!\\
&=\prec e_1=\arbreda, e_2=\arbreda+\arbredb, e_3=\arbreda+\arbredc, e_4=\arbreda+\arbredb+3 \arbredc+\arbredd, e_5=\arbreda+\arbrede,\\
& \hskip 8mm e_6=\arbreda+\arbredb+\arbrede+\arbredz, e_7=\arbreda+\arbredc+\arbredf, e_8=\arbreda+\arbredc+2 \arbrede+\arbredf+\arbredg,\\
& \hskip 8mm e_9=\arbreda+\arbredb+3 \arbredc+\arbredd+4 \arbrede+4 \arbredz+3 \arbredf+6 \arbredg+\arbredh\!\succ\!\!. 
\end{align*}

Now, for any homogeneous component $\Cal{T}_{\!\!n}$, each vector in the monomial basis described above is defined as a monomial $m(\racine, \to)$ of the tree with one vertex "$\racine$" multiplied (by itself) using the pre-Lie grafting $"\to"$ with the parentheses. This monomial in turn determines two monomials in the algebras $(\Cal{T}^{pl},\lbutcher)$ and $(\Cal{T}, \butcher)$ respectively. One of these monomials is obtained by replacing the grafting $"\to"$ by the left Butcher product $"\,\lbutcher"$, which induces a planar rooted tree $\tau$. The other monomial is  deduced by replacing the product $"\to"$ by the usual Butcher product, which in turn defines a (non-planar) rooted tree $t$.
\begin{defn}\label{good basis}
A monomial basis for a free pre-Lie algebra is said to be a "tree-grounded" monomial basis if we obtain the Chapoton-Livernet tree basis when we replace the pre-Lie product in each monomial in this basis by the Butcher product $"\butcher"$. For any positive integer $n$, a monomial basis of $\Cal{T}_n$ will also be called tree-grounded if this property holds in $\Cal{T}_{\!\!n}$. 
\end{defn}
\begin{exam}
In the space of all (non-planar) rooted trees $\Cal T$, the homogeneous component $\Cal{T}_4$ has four types of monomial bases, which are:
\begin{align*}
\Cal{B}_1&=\big\{((\racine\to\!\racine)\to\!\racine)\to\!\racine, (\racine\to\!(\racine\to\!\racine))\to\!\racine, (\racine\to\!\racine)\to\!(\racine\to\!\racine), \racine\to\!(\racine\to\!(\racine\to\!\racine))\,\big\}\\ 
&=\big\{e_1=\arbreca, e_2=\arbreca+\arbrecb, e_3=\arbreca+\arbrecc, e_4=\arbreca+\arbrecb+3 \arbrecc+\arbrecd \big\}.\\
\Cal{B}_2&=\big\{((\racine\to\!\racine)\to\!\racine)\to\!\racine, (\racine\to\!(\racine\to\!\racine))\to\!\racine, \racine \to((\racine\to\!\racine)\to\!\racine), \racine\to\!(\racine\to\!(\racine\to\!\racine))\,\big\}\\ 
&=\big\{e_1=\arbreca, e_2=\arbreca+\arbrecb, e_3=\arbreca+\arbrecb+\arbrecc, e_4=\arbreca+\arbrecb+3 \arbrecc+\arbrecd \big\}.\\
\Cal{B}_3&=\big\{((\racine\to\!\racine)\to\!\racine)\to\!\racine, (\racine\to\!\racine)\to (\racine\to\!\racine), \racine \to((\racine\to\!\racine)\to\!\racine), \racine\to\!(\racine\to\!(\racine\to\!\racine))\,\big\}\\ 
&=\big\{e_1=\arbreca, e_2=\arbreca+\arbrecc, e_3=\arbreca+\arbrecb+\arbrecc, e_4=\arbreca+\arbrecb+3 \arbrecc+\arbrecd \big\}.\\
\Cal{B}_4&=\big\{(\racine\to\!(\racine\to\!\racine))\to\!\racine, (\racine\to\!\racine)\to (\racine\to\!\racine), \racine \to((\racine\to\!\racine)\to\!\racine), \racine\to\!(\racine\to\!(\racine\to\!\racine))\,\big\}\\ 
&=\big\{e_1=\arbreca+\arbrecb, e_2=\arbreca+\arbrecc, e_3=\arbreca+\arbrecb+\arbrecc, e_4=\arbreca+\arbrecb+3 \arbrecc+\arbrecd \big\}.
\end{align*}
We find that the monomial bases $\Cal{B}_1$ and $\Cal{B}_2$ are tree-grounded monomial bases of $\Cal{T}_4$, because replacing the pre-Lie grafting $"\to"$ by the Butcher product $"\butcher"$ gives back the tree basis $\big\{\arbreca, \arbrecb, \arbrecc, \arbrecd \big\}$. But one can note that the bases $\Cal{B}_3$ and $\Cal{B}_4$ are not tree-grounded.\footnote{We thank the referee for having suggested this example to us.}  
\end{exam}

\begin{lem}
A monomial basis for the free pre-Lie algebra with one generator is tree-grounded if and only if it comes from a section map $S$ according to the linear map $\widetilde{\Psi}_S$.
\end{lem}
\begin{proof}
For any element $x=m(\racine, \to)$ of some tree-grounded monomial basis, let $S(t)=m(\racine, \lbutcher)$ where $t=m(\racine, \butcher)$ is the lower-energy term of $x$. By Definition \ref{good basis}, these lower-energy terms form  a basis of $\Cal T$, hence $S$ is uniquely defined that way, and it is a section of $\pi$, as in the following diagram:
\diagrama{
\xymatrix{
m(\racine,\lbutcher) \in \Cal {T}^{pl} \ar[r]^{\Psi}\ar@<12mm>@{->>}[d]^{\pi}
&\Cal T^{pl} \ni m(\racine, \searrow)\ar@<-9mm>@{->>}[d]^{\pi}\\
t=m(\racine, \butcher) \in \Cal{T} \ar@{.>}[r]_{{\widetilde\Psi}_S}\ar@<-10mm>@{.>}[u]^S & \Cal{T}\ni x=m(\racine, \to)}}
On the other hand, any monomial basis induced by a section map $S$ is obviously a tree-grounded monomial basis. 
\end{proof}
 
\begin{lem}\label{Agrachev-Gamkrelidze monomial}
The Agrachev-Gamkrelidze monomial bases are tree-grounded.
\end{lem}
\begin{proof}
From the construction of Agrachev-Gamkrelidze monomial bases, and using the presentation of the free pre-Lie algebra in terms of rooted trees (see Theorem \ref{d-generators}), with one generator, the formula \eqref{formula A} can be written as:
$$\Cal{A}_n \tilde{=} \Cal{U}_{n-1},\ \forall n \geq 1$$
such that for a homogeneous component $\Cal{A}_n$, the monomial basis in \eqref{equation B} becomes: 
$$\left\{ x^{e_1}_{\!j_{1}} \to \big (x^{e_2}_{\!j_{2}} \to (\cdots \to (x^{e_r}_{\!j_{r}}\to \racine) \cdots)\big )\ \big|\ \sum^{r}_{k=1}{j_k}=n-1,\ x^{e_1}_{\!j_{1}} \geq \cdots \geq x^{e_r}_{\!j_{r}} \right\}.$$

The monomial basis for $\Cal{A}_1$, namely $\{\racine\}$, is obviously tree-grounded in the sense of Definition \ref{good basis}. Suppose, by the induction hypothesis, that the monomial basis $\{x_j^{e_1}, \ldots, x_j^{e_j}\}$ is a tree-grounded basis of $\Cal{A}_j$, for $j=1, \ldots, n-1$. Consider the corresponding lower-energy terms $t_j^{e_1}, \ldots, t_j^{e_j}$ obtained by replacing the grafting "$\to$" by the Butcher product $"\butcher"$ in each monomial. The lower-energy term of the monomial
\begin{equation}\label{monomial A}
x^{e_1}_{\!j_{1}} \to \big (x^{e_2}_{\!j_{2}} \to (\cdots \to (x^{e_r}_{\!j_{r}}\to \racine) \cdots)\big )
\end{equation}
is given by :
\begin{align*}
& t^{e_1}_{\!j_{1}} \butcher \big (t^{e_2}_{\!j_{2}} \butcher (\cdots \butcher (t^{e_r}_{\!j_{r}}\butcher \racine) \cdots)\big ) \\
&= B_{\!_{+}}(t_{j_1}^{e_1} \ldots t_{j_r}^{e_r}). 
\end{align*}

Hence we recover the tree basis of $\Cal{A}_n$ by taking the lower-energy term of each monomial \eqref{monomial A}, thus proving Lemma \ref{Agrachev-Gamkrelidze monomial}. 
\end{proof} 
As a particular case of our general construction, an Agrachev-Gamkrelidze monomial basis, by means of the isomorphism \eqref{f}, gives rise to some particular section $S$. Conversely, any section $S$ of $\pi$ defines a tree-grounded monomial basis for the free pre-Lie algebra $(\Cal{T}, \to)$. For any integer $n\geq 1$, the matrix of the coefficients of the tree-grounded monomial of $\Cal{T}_n$ associated with the section $S$ is exactly the matrix $[\![\beta_S(s, t)]\!]_{\,s,\,t \in \Cal{T}_{\!n}}^{}$ described in Corollary \ref{cor 5} in the preceding section.\\

\paragraph{\textbf{Acknowledgments}}{I would like to thank D. Manchon for his continuous support to me throughout the writing of this paper. I also thank F. Patras and K. Ebrahimi-Fard for valuable discussions and comments. Finally, I am grateful to the referee for her/his careful evaluation of this work and for all suggestions.}

%%%%%%%%%%%%%%%%%%%%%%%%%%%%%%%%%%%%%%%%%%%%%%%%%%%%%%%%%%%%%%%%%%%
%%%%%%%%%%%%%%%%%%%%%%%%%%%%%%%%%%%%%%%%%%%%%%%%%%%%%%%%%%%%%%%%%%%


\begin{thebibliography}{abcdsfgh}
{\small{
\bibitem{AR81}A. Agrachev, R. Gamkrelidze, {\sl Chronological algebras and nonstationary vector fields\/}, J. Sov. Math. 17 Nol, 1650-1675 (1981).
\bibitem{CB00}Ch. Brouder, {\sl Runge-Kutta methods and renormalization\/}, Europ. Phys. J. C12, 521-534 (2000).
\bibitem{C.B04}Ch. Brouder, {\sl Trees, Renormalization and differential equations\/}, BIT Numerical Mathematics 44, 425-438 (2004).
\bibitem{D.B06}D. Burde, {\sl Left symmetric algebras, or pre-Lie algebras in geometry and physics\/}, Cent. Eur. J. Math. 4(3), 323-357 (2006).
\bibitem{A.C57}A. Cayley, {\sl On the Theory of Analytical Forms called Trees\/}, Philosophical Magazine 13, 172-176 (1857).
\bibitem{P.C11}P. Cartier, {\sl Vinberg Algebras\/}, Lie groups and combinatorics, Clay Mathematical Proceedings 11, 107-126 (2011).
\bibitem{CL01}F. Chapoton, M. Livernet, {\sl Pre-Lie algebras and the rooted trees operad\/}, Internat. Math. Res. Notice 8, 395-408 (2001), arXiv: math/0002069.
\bibitem{CK98}A. Connes and D. Kreimer, {\sl Hopf algebras, renormalization and noncommutative geometry\/}, Comm. Math. Phys. 199, 203-242 (1998).
\bibitem{J.D74}J. Dixmier, {\sl Alg\`ebres Enveloppantes\/}, Paris: Gauthier-Villars, France (1974).
\bibitem{AC02}A. Dzhumadil'daev, C. L\"ofwall, {\sl Trees, free right-symmetric algebras, free Novikov algebras and identities\/}, Homology, Homotopy and Appl. 4(2), 165-190 (2002).
\bibitem{EM}K. Ebrahimi-Fard, D. Manchon, {\sl On an extension of Knuth's rotation correspondence to reduced planar trees \/}, Journal of Noncommutative Geometry (to appear).
\bibitem{PT09}Ph. Flajolet, T. Sedgewick, {\sl Analytic combinatorics\/}, Cambridge Univ. Press (2009).
\bibitem{M.G63}M. Gerstenhaber, {\sl The cohomology structure of an associative ring\/}, Ann. Math. 78, No. 2, 267-288 (1963).
\bibitem{D.K68}D. E. Knuth, {\sl The art of computer programming I. Fundamental algorithms\/}, Addison-Wesley (1968).
\bibitem{J.K61}J. L. Koszul, {\sl Domaines born\'es homog\`enes et orbites de groupes de transformations affines\/}, Bull. Soc. Math. France 89, No. 4, 515-533 (1961).
\bibitem{D.M}D. Manchon, {\sl Algebraic Background for Numerical Methods, Control Theory and Renormalization\/}, Proc. Combinatorics and Control, Benasque, Spain, 2010 (to appear).
\bibitem{JM04}J.-F. Marckert, {\sl The rotation correspondence is asymptotically a dilatation\/}, Random Structures \& Algorithms 24, Issue 2, 118-132 (2004).
\bibitem{HW08}H. Munthe-Kaas, W. Wright, {\sl On the Hopf Algebraic Structure of Lie Group Integrators\/}, Found. Comput. Math. 8, No. 2, 227-257 (2008).
\bibitem{Ch.R93}Ch. Reutenauer, {\sl Free Lie algebras\/}, Oxford University Press, New York(US) (1993).
\bibitem{DS94}D. Segal, {\sl Free Left-Symmetric Algebras and an Analogue of the Poincar\'e- Birkhoff-Witt Theorem \/}, J. Alg. 164, 750-772 (1994).
\bibitem{S}N. J. A. Sloane, {\sl The On-Line Encyclopedia of Integer Sequences\/}, oeis.org.
\bibitem{E.V63}E. B. Vinberg, {\sl The Theory of homogeneous convex  cones\/}, Transl. Moscow Math. Soc. 12, 340-403 (1963).
}}
\end{thebibliography}
\end{document}